\documentclass[12pt, reqno]{amsart}
\usepackage[margin=1.2in]{geometry}                
\usepackage{graphicx}
\usepackage{amssymb,amscd,amsthm}
\usepackage{epstopdf}
\usepackage{tikz}
\usepackage[utf8]{inputenc}
\usepackage[colorlinks=true, pdfstartview=FitV, linkcolor=blue, citecolor=blue, urlcolor=blue]{hyperref}
\numberwithin{equation}{section}
\usepackage{tikz-cd}
\usepackage[nameinlink]{cleveref}
\usepackage{comment}

\def\a{\ensuremath{\alpha}}
\def\ua{\ensuremath{\underline{\alpha}}}
\def\b{\ensuremath{\beta}}
\def\ub{\ensuremath{\underline{\beta}}}

\def\e{\ensuremath{\varepsilon}}

\newcommand{\wa}{\widehat{\alpha}}
\newcommand{\wb}{\widehat{\beta}}
\newcommand{\wrho}{\widehat{\rho}}
\newcommand{\ba}{\mathbf{a}}
\newcommand{\bb}{\mathbf{b}}
\newcommand{\be}{\mathbf{e}}

\newcommand{\fm}{{\mathfrak m}}

\newcommand{\N}{{\mathbb N}}
\renewcommand{\P}{{\mathbb P}}
\newcommand{\Q}{{\mathbb Q}}
\newcommand{\R}{{\mathbb R}}
\newcommand{\Z}{{\mathbb Z}}

\newcommand{\kk}{\mathbb{K}}
\newcommand{\m}{\mathbf{m}}
\newcommand{\contract}{\bullet}

\newcommand{\cI}{{\mathcal I}}

\newcommand{\cL}{{\mathcal L}}
\newcommand{\I}{{\mathcal I}}
\renewcommand{\L}{{\mathcal L}}

\newcommand{\SP}{{\rm SP}}

\def\Ass{\operatorname{Ass}}

\DeclareMathOperator{\reg}{reg}
\DeclareMathOperator{\soc}{end}

\DeclareMathOperator{\Sym}{Sym}

\def\id{{\operatorname{id}}}

\theoremstyle{plain} 

\newtheorem{thm}{Theorem}[section]

\newtheorem*{introthm*}{Theorem}
\newtheorem{question}[thm]{Question}
\newtheorem{cor}[thm]{Corollary}
\newtheorem{lem}[thm]{Lemma}
\newtheorem{prop}[thm]{Proposition}
\newtheorem{conv}[thm]{Convention}
\newtheorem{conj}[thm]{Conjecture}

\theoremstyle{definition}
\newtheorem{defn}[thm]{Definition}

\newtheorem{ex}[thm]{Example}

\newtheorem{rem}[thm]{Remark}

\theoremstyle{remark}

\newcommand{\uao}{\overleftarrow{\ua}}
\newcommand{\uau}{\overrightarrow{\ua}}
\newcommand{\ao}{\overleftarrow{\alpha}}
\newcommand{\au}{\overrightarrow{\alpha}}
\newcommand{\ubo}{\overleftarrow{\ub}}

\newcommand{\bo}{\overleftarrow{\beta}}

\newcommand{\so}{\overleftarrow{s}}

\renewcommand{\L}{\mathcal{L}}

\title[Duality for asymptotic invariants]{Duality for asymptotic invariants \\of graded families}
\author{Michael DiPasquale}
\address{Department of Mathematics and Statistics, University of South Alabama}
\email{mdipasquale@southalabama.edu}
\author{Th\'{a}i Th\`{a}nh Nguy$\tilde{\text{\^e}}$n}
\address{Department of Mathematics, Tulane University \\
and University of Education, Hue University, 34 Le Loi St., Hue, Viet Nam}
\email{tnguyen11@tulane.edu}
\author{Alexandra Seceleanu}
\address{Department of Mathematics, University of Nebraska-Lincoln}
\email{aseceleanu@unl.edu}


\begin{document}

\begin{abstract}
The starting point of this paper is a duality for sequences of natural numbers which, 
under mild hypotheses, interchanges subadditive and superadditive sequences and inverts their asymptotic growth constants.

We are motivated to explore this sequence duality since it arises naturally in at least two important algebraic-geometric contexts.  The first context is Macaulay-Matlis duality, where the sequence of initial degrees of the family of symbolic powers of a radical ideal is dual to  the sequence of Castelnuovo-Mumford regularity values of a quotient by ideals generated by powers of linear forms. This philosophy is drawn from an influential paper of Emsalem and Iarrobino. We generalize this duality to differentially closed graded filtrations of ideals.

In a different direction, we establish a duality between the sequence of Castelnuovo-Mumford regularity values of the symbolic powers of certain ideals and a geometrically inspired sequence we term the jet separation sequence. We show that this duality underpins the reciprocity between two important geometric invariants: the multipoint Seshadri constant and the asymptotic regularity of a set of points in projective space.


%
%
%
\end{abstract}

\maketitle
\setcounter{tocdepth}{1}
\tableofcontents

\vspace{-0.5em}
\section{Introduction}

As Michael Atiyah \cite{Atiyah} points out,
``Duality in mathematics is not a theorem, but a principle''. Indeed, forms of duality occur in all branches of mathematics manifesting in ways specific to the subject area.
In this paper we study manifestations of duality which take effect primarily in an algebraic-geometric context. More precisely, our starting point is a notion of duality for sequences of natural numbers. This prompts the question of determining the dual sequences for certain numerical sequences which occur in commutative algebra, for example, the sequence of initial degrees of a graded family of homogeneous ideals, or the sequence of Castelnuovo-Mumford regularity values of a family of ideals. Our techniques allow to relate the asymptotic growth factors of these sequences to those of the dual sequences. We explore this theme in contexts where these asymptotic growth factors carry significant meaning.

At the level of numerical sequences we single out two transformations which act on nondecreasing sequences of integers. Given a sequence $\ua=\{\a_n\}_{n\in\N}$, the transformed sequences are as follows
 \begin{eqnarray*}
 \ao_n=\inf\{d\mid \alpha_d\geq n\},\\
\au_n=\sup\{d\mid \alpha_d\leq n\}.
\end{eqnarray*}
It turns out that these transformation are mutual inverses \cite{Johnson11}. If furthermore $\ua$ is either a subadditive or superadditive sequence (see \Cref{def:subsuper}) then it has a well-defined asymptotic growth factor $\wa=\lim_{n\to \infty}\frac{\a_n}{n}$. The above transformations interchange the classes of subadditive and superadditive sequences. Moreover, under these hypotheses,  we are able to derive the following reciprocation formulas for the respective the asymptotic growth factors in \Cref{lem:DualityProperties}:
\[
\widehat{\au} = \widehat{\a}^{-1} \quad \text{ and} \quad \widehat{\ao} = \widehat{\a}^{-1}.
\]
In fact, we 
generalize the above transformations as well as the reciprocation formulas in \Cref{thm:relativeduality}. The more technical statement of this result is relegated to \cref{s:sequences}.

We apply the duality principle described above to several numerical sequences. Our interest in such sequences is spurred by the study of the family of symbolic powers $\{I^{(d)}\}_{d\in\N}$ of a homogeneous ideal $I$. In the case when $I$ is the defining ideal of an algebraic variety $X$, this family features prominently in algebraic geometry by encoding the set of functions vanishing to higher order on $X$. In commutative algebra, the symbolic powers have been studied most recently by means of comparison with the family of ordinary powers $\{I^n\}_{n\in\N}$; see \cite{ELS, HoHu, MaSchwede, BH10}.

A sequence of interest in this area of study is given by the initial degrees for the symbolic power ideals. Its asymptotic growth factor, dubbed the {\em Waldschmidt constant}, is $\wa(I)=\lim_{d\to \infty}\frac{\alpha(I^{(d)})}{d}$. It is well-known that this sequence is subadditive.  The same is true for any sequence that results from applying a discrete valuation to a graded family of ideals (see \Cref{lem:WaldSub}). Taking this more general perspective leads to considering valuative sequences for any discrete valuation  $\nu$
\[
\beta^{\nu}_n=\beta^{\nu}_n(I)=\sup\{d:\nu(I^{(d)}) < \nu(I^n)\}.
\]
In \Cref{lem:valuationLimit} we apply our duality results to relate the growth factor $\widehat{\beta^{\nu}}$ of this sequence to those of the sequences $\{\nu(I^{(d)})\}_{d\in\N}$ and $\{\nu(I^n)\}_{n\in\N}$. 
This has consequences on  the {\em containment problem} between the ordinary and symbolic powers of $I$.  Building on \cite{DFMS, DD21}, we show that there exists a valuation $\nu$ for which the asymptotic growth factor of $\beta^{\nu}$ recovers the {\em asymptotic resurgence} of \cite{GHV13}.
%

In \cref{sec:InverseSystems} we study the dual of a sequence closely related to the initial degree sequence of the family of symbolic powers. In this pursuit, we are led to consider a notion of inverse systems which dates back to Macaulay \cite{Macaulay}. Emsalem and Iarrobino determined in an influential paper \cite{EI95} the inverse system for the symbolic powers of a radical ideal. We generalize their results by introducing a new notion of differentially closed graded filtrations of ideals for which the inverse systems behave particularly well.  Examples of differentially closed graded filtrations include many families of powers of a homogeneous ideal (differential, ordinary, symbolic, and integral Frobenius powers) and any family obtained by intersecting these.

\begin{introthm*}[\Cref{thm:differentiallyclosed} and \Cref{thm:InverseSystemDualityGeneral}]
Suppose $\I=\{I_n\}_{n\in\N}$ is a differentially closed graded filtration of proper ideals in $R=\kk[x_0,\ldots,x_N]$.  Let  $\mathcal{D}=\bigoplus_{i\geq 0} {\rm Hom}(R_i,\kk)$, equipped with the structure of a divided power algebra.  For each $s\in\N$ put 
\[
\L^s(\cI):=\bigoplus_{d\ge s+1} \left(I_{d-s}^{\perp}\right)_d \subseteq \mathcal{D}.
\]
Then $\L^s(\cI)$ is an ideal of $\mathcal{D}$ for each $s\in\N$.  The sequence $\alpha_n = \alpha(I_n)$ is subadditive and $\beta_s=\sup\{d:(\mathcal{D}/\L^s(\I))_d\neq 0\}$ is superadditive.  Assuming that the sequence $\{\alpha_n-n\}_{n\in\N}$ is not bounded above, we have
\[
\wb=\frac{\wa}{\wa-1} \mbox{  and  } \wa=\frac{\wb}{\wb-1}.
\]
\end{introthm*}

Several forms of algebraic duality manifest themselves in the setup above. The degreewise vector space duality between $R$ and $\mathcal{D}$ manifests itself via apolarity (orthogonality). The inverse systems considered in  \cref{sec:InverseSystems} are a form of Matlis duality. Finally the projective duality between points in $p=(p_0:\cdots:p_N)\in \P^N$ and linear forms $L_p=p_0x_0+\cdots+p_nx_n\in R$ yields a celebrated description of $\L^s(\cI)$  when $\mathcal{I}$ is the family of symbolic powers for the defining ideal of a projective variety; see \Cref{ex:LsForSymbolicPowers}.

In contrast to the above setting where the algebraic duality is more evident while the numerical duality of asymptotic invariants is more elusive, we study a different setup where duality of asymptotic invariants has been observed before (see \cite[\S 5.1]{Laz04}), but the underpinning reasons have not previously been discovered. 

\begin{introthm*}[\Cref{prop:superadd} and \Cref{thm:areg}]
Let $I$ be the defining ideal of a set $X$ of $r\geq 2$  points in $\P^N$. Set $s_d=s(X,d-1)$ to be the jet separation sequence of $X$ (\Cref{def:jets}) and $r_k=\reg(I(X)^{(k+1)})$ the sequence of Castelnuovo-Mumford regularites for the symbolic powers of $I(X)$. There is a duality between these sequences
\[
 s=\overrightarrow{r} \text{ and } r=\so.
\]
This duality underlies the following identity relating the Seshadri constant $\e(X)$ (see \Cref{def:Seshadri}) of $X$ and the asymptotic regularity of $X$
\[
\e(X)=\lim_{d\to\infty}\frac{s_d}{d}= \left(\lim_{k\to\infty}\frac{I(X)^{(k)}}{k}\right)^{-1}= :\widehat{\reg}(I(X))^{-1}.
\]
\end{introthm*}

In \cref{s:Nagata} we take the opportunity to revisit the celebrated conjectures of Nagata and Iarrobino regarding linear systems of polynomials vanishing to higher order at a finite set of points in projective space. We give homological reformulations for these conjectures based on the results discussed above. This leads into further open problems presented in the final \cref{s:open}.

\section{Duality for numerical sequences}\label{s:1}
\label{s:sequences}
In this paper the set $\N$ of natural numbers does not include 0.

The purpose of this section is to study duality of sequences of natural numbers. To define this duality we first generalize two operations on sequences introduced in \cite{Johnson11}. These operations are discrete analogues for the notion of  pseudo-inverse functions described in \cite{Osterdal06}.

\begin{defn}
\label{def:overunder}
Given sequences $\ua=\{\alpha_d\}_{d\in \N}$ and $\ub=\{\b_n\}_{n\in \N}$ of real numbers 
define new sequences $\uao^\b, \uau^\b$ 
associated to the pair $\ua, \ub$ in the following manner, where we allow that $\ao^\b_n, \au^\b_n\in \N\cup\{-\infty, \infty\}$ (by convention $\sup(\emptyset)=-\infty, \inf(\emptyset)=\infty$):
 \begin{eqnarray*}
 \ao^\b_n=\inf\{d \in \N \mid \alpha_d\geq \b_n\},\\
\au^\b_n=\sup\{d \in \N \mid \alpha_d\leq \b_n\}.
\end{eqnarray*}
Setting $\id_n=n$ yields the two particularly important  sequences previously studied in \cite{Johnson11}, for which we use the shortened notation $\uao^\id=\uao$ and $\uau^\id=\uau$. They are given by
 \begin{eqnarray*}
 \ao_n=\inf\{d\mid \alpha_d\geq n\},\\
\au_n=\sup\{d\mid \alpha_d\leq n\}.
\end{eqnarray*}
\end{defn}

In the remainder of the paper we will be interested  in situations when the sequences $\ua, \ub$ consist of natural numbers and for all $n\in \N$ they yield $\ao^\b_n\in \N$ and $\au^\b_n\in \N$. 
\begin{ex}
If $\ua$ is a sequence of natural numbers there are identities $\overrightarrow{\underline{\id}}^\a=\overleftarrow{\underline{\id}}^\a=\ua$.  
\end{ex}

We shall be interested in applying the transformations in \Cref{def:overunder} to subadditive and superadditive sequences respectively. We now recall these notions.

\begin{defn}
\label{def:subsuper}
A sequence of real  numbers $\ua=\{\alpha_n\}_{n\geq n_0}$ for some $n_0\in\N$ is called 
\begin{itemize}
\item {\bf subadditive} if it satisfies $\alpha_{i+j} \leq \alpha_i+\alpha_j$ for all $i,j\geq n_0$.
\item {\bf superadditive} if it satisfies $\alpha_i+\alpha_j\leq \alpha_{i+j}$ for all $i,j\geq n_0$.
\end{itemize}
\end{defn}

 Fekete's lemma \cite{Fekete} guarantees the existence of  $\wa=\lim_{n\to\infty} \frac{\a_n}{n}$ for any subadditive  or superadditive sequence of real numbers $\ua=\{\a_n\}_{n\in \N}$, allowing for the value of the limit to be $-\infty$ in the subadditive case and $\infty$ in the superadditive case respectively.  In the subbaditive case, the value of the limit coincides with $\inf_{n\in \N} \frac{\a_n}{n}$ and in the superadditive case with $\sup_{n\in \N} \frac{\a_n}{n}$. 
 
 \begin{defn}
 Given a subadditive  or superadditive sequence of real numbers $\ua=\{\a_n\}_{n\geq n_0}$,  the {\bf asymptotic growth factor} 
 of $\ua$ is the value of the limit 
 \[
 \wa=\lim_{n\to\infty} \frac{\a_n}{n}\in \R\cup\{-\infty, \infty\}.
 \]
 If additionally $\ua$ consists of natural numbers, we have  $\wa\in \R_{\geq 0} \cup\{ \infty\}$.  
 \end{defn}

 We now arrive at our first main result. It shows, how the transformations in \Cref{def:overunder} interact with the classes of subadditive and superadditive sequences and how they transform the respective asymptotic growth factors. In the statement we adopt the conventions that $r/0=\infty$ and $r/\infty=0$ for $r\in \R_{>0}$ and $\infty/0=\infty$, $0/\infty=0$.
 
 \begin{thm}
\label{thm:relativeduality}
Consider sequences $\ua, \ub$ of positive real numbers such that $\ua$ is subadditive  and $\ub$ is superadditive and for each $n\in\N$ we have $\au^\b_n\in\N$ and $\bo^\a_n\in\N$. Then
\begin{enumerate}
\item  the sequence $\uau^\b$ is superadditive and satisfies $\widehat{\au^\b}=\widehat{\b}/\widehat{\a}$.
\item the sequence $\ubo^\a$ is subadditive and satisfies $\widehat{\bo^\a}=\widehat{\a}/\widehat{\b}$.
\end{enumerate}
\end{thm}
\begin{proof}
(1) Let $m,n\in\N$ and set $d=\au^\b_m$ and $d'=\au^\b_n$. By definition we have $\a_d\leq \b_m$ and $\a_{d'}\leq \b_n$ whence we deduce using subadditivity of $\ua$ and superadditivity of $\ub$
\[
\a_{d+d'}\leq \a_d+\a_{d'}\leq \b_m+\b_n\leq \b_{m+n}.
\]
It follows that $\au^\b_{m+n}\geq d+d'=\au^\b_m+\au^\b_n$, establishing superadditivity for $\uau^\b$.

Assume first that $\wa\neq 0$ and $\wb\in \R$ (i.e., $\wb\neq \infty$). 
Since $\uau^\b$ is superadditive, we have 
\begin{equation}
\label{eq:sup}
\widehat{\au^\b}=\sup_{n\in \N}\left\{ \frac{\au^\b_n}{n} \right\}=\sup \left \{\frac{d}{n} \mid \a_d\leq \b_n \right \}.
\end{equation}
The identities $\wa=\lim\limits_{d\to\infty}\left\{ \frac{\a_d}{d} \right\}=\inf\limits_{d\in\N}\left\{ \frac{\a_d}{d} \right\}$ and $\wb=\lim\limits_{n\to \infty}\left\{ \frac{\b_n}{n} \right\}=\sup\limits_{n\in\N}\left\{ \frac{\b_n}{n} \right\}$ yield 
\[
\frac{\wb}{\wa}=\sup\limits_{n,d\in \N}\left\{ \frac{\b_n}{\a_d}\cdot \frac{d}{n} \right\},
\]
whence we deduce that $\frac{\wb}{\wa}\geq \frac{\b_n}{\a_d}\cdot \frac{d}{n}\geq  \frac{d}{n}$ whenever $\a_d\leq \b_n$.
Combining this with \eqref{eq:sup} we arrive to the conclusion $\frac{\wb}{\wa}\geq \widehat{\au^\b}$.

To establish the converse inequality it suffices to show that for all $n,d\in\N$ with $\frac{d}{n}<\frac{\wb}{\wa}$ we have  $\frac{d}{n}\leq  \widehat{\au^\b}$. Assuming that $\frac{d}{n}<\frac{\wb}{\wa}$ or equivalently that $\frac{d}{n} \cdot \wa<{\wb}$  and writing $\frac{d}{n} \cdot \wa=\lim\limits_{t\to \infty}\frac{d}{n}\cdot\frac{\a_{dt}}{dt}$ and $\wb=\lim\limits_{t\to \infty}\frac{\b_{nt}}{nt}$ allows to conclude that for $t\gg 0$ we have 
\begin{equation}
\label{eq:dtvsnt}
\frac{d}{n}\cdot\frac{\a_{dt}}{dt}<\frac{\b_{nt}}{nt}, \text{ that is, } \a_{dt}<\b_{nt} \text{ for } t\gg0.
\end{equation}
In view of the above inequality, \eqref{eq:sup} yields $\widehat{\au^\b}\geq \frac{dt}{nt}$ for $ t\gg0$, which leads to the desired conclusion  $\widehat{\au^\b}\geq \frac{d}{n}$. This concludes the proof of the claim  $\widehat{\au^\b}=\widehat{\b}/\widehat{\a}$.

Now we treat the cases $\wa=0$ and $\wb=\infty$. In both of these situations our convention yields $\wb/\wa=\infty$. For arbitrary $d,n\in \N$ the inequality $\frac{d}{n} \cdot \wa<{\wb}$ is satisfied, therefore the same argument as in \eqref{eq:dtvsnt} yields $\widehat{\au^\b}\geq \frac{d}{n}$ for all $d,n\in \N$. It follows that $\widehat{\au^\b}=\infty=\wb/\wa$, as claimed.

%
%

(2) The second part is entirely analogous to the first.
\end{proof}

Specializing the previous theorem to the case when one of the sequences involved is $\underline{\id}$ allows for a result that better portrays the duality of the transformations $\uau$ and $\uao$.  To obtain a true duality theory one must restrict to the case when the input sequence $\ua$ is  a sequence of natural numbers unbounded above. Specifically, the next result, which constituted the starting point of our project, shows that the transformations $\uau, \uao$ are mutual inverses and  interchange the classes of subadditive and superadditive sequences, that, when restricted to these classes of sequences, the transformations $\uau, \uao$ reciprocate the respective asymptotic growth factors.

 For the next result we utilize the convention that $0^{-1}=\infty$ and $\infty^{-1}=0$.

\begin{thm}
\label{lem:DualityProperties}
Let $\ua$ be a nondecreasing sequence of natural numbers. 
\begin{enumerate}
\item There are identities
$\overrightarrow{\uao}=\ua \text{ and } \overleftarrow{\uau}=\ua$.
\item 
If $\ua$ is increasing, then there are identities
$\overrightarrow{\uau}=\ua$  and  $\overleftarrow{\uao}=\ua$.

\item  If $\ua$ is subadditive  then $\{\au_n\}_{n\geq \a_1}$ is nondecreasing superadditive with  $\widehat{\au} = \widehat{\a}^{-1}$.
\item  If $\ua$ is superadditive, then $\uao$ is nondecreasing subadditive with $\widehat{\ao} =\widehat{\a}^{-1}$.
\end{enumerate}
\end{thm}

\begin{proof}
%

Assertion (1) as well as the assertions that whenever $\ua$ is superadditive,  $\uao$ is subadditive and whenever $\ua$ is superadditive, then $\uao$ is subadditive are shown in \cite[Corollary 2.8]{Johnson11}. 

For part (2), note that  whenever $\ua$ is increasing the following hold
\begin{eqnarray}
\label{eq:dual}
 \au_{\a_n}=  \sup\{t: \a_t\leq \a_n\}=  n\\
\label{eq:dual'}
 \ao_{\a_n}=  \inf\{t: \a_t\geq \a_n\}=  n
\end{eqnarray}
 Given this, we obtain by applying equation \eqref{eq:dual} for  $\ua, \uau$ the following identity
 \[
\overrightarrow{\au}_n= \overrightarrow{\au}_{ \au_{\a_n}}=\a_n.
 \]
Similarly,  applying equation \eqref{eq:dual'} for each of the sequences $\ua, \uao$ we obtain
 \[
 \overleftarrow{\ao}_n= \overleftarrow{\ao}_{ \ao_{\a_n}}=\a_n.
 \]

The remaining assertions of the theorem regard the asymptotic growth factors. These can be recovered from  \Cref{thm:relativeduality}  as follows: first, observe that setting $\id_n=n$ in yields $\widehat{\id}=1$, $\uao^\id=\uao$ and $\uau^\id=\uau$. Note that if $\ua$ is nondecreasing, so are $\uau$ and $\uao$ by definition.

If $\ua$ is a superadditive sequence of natural numbers then it is unbounded above as $\a_n\geq n\a_1\geq n$. It follows that  $\uao_n\in \N$ for all $n\in\N$ whenever the sequence $\ua$ is superadditive. If $\ua$ is a nondecreasing sequence it follows that $\uau_n\in \N$ for all $n\geq \a_1$.

Since  $\underline{\id}$ is both subadditive and superadditive, setting $\b=\id$ in part (1) of \Cref{thm:relativeduality} yields for subbadditive $\ua$ that $\widehat{\uau}=\wa^{-1}$ as in part (3) of \Cref{lem:DualityProperties}, and setting $\ua=\id$ in part (2) of \Cref{thm:relativeduality} yields for superadditive $\ub$ that $\widehat{\ubo}=\wb^{-1}$ as in part (4) of \Cref{lem:DualityProperties}.

\end{proof}

\begin{ex}
In the absence of the hypothesis that $\ua$ is nondecreasing, it need not be true that $\overleftarrow{\uau}=\ua$. Consider the sequence 
$
\a_n=\begin{cases} 
n & \text{ if } n \text{ is odd}\\
n/2 & \text{ if } n \text{ is even}.
\end{cases}
$
Then we compute
\begin{center}
\begin{tabular}{c|cccccc}
n & 1& 2 & 3 & 4 & 5  \\
\hline
$\a_n$ & 1 & 1 & 3 & 2 & 5 \\
$\au_n$ & 2& 4 & 6 & 8 & 10 \\
$\overleftarrow{\au}_n$ & 1& 1 & 2 & 2 & 3.
\end{tabular}
\end{center}

Likewise, in the absence of the hypothesis that $\ua$ is increasing, it need not be true that $\overrightarrow{\uau}=\ua=\ua$. Consider $\a_n=\lceil n/2\rceil$. Then $\ua$ is subadditive and non decreasing (but not increasing), $\au_n=2n$ and $\overrightarrow{\au}=\ua_n=\lfloor n/2\rfloor$. 

\end{ex}

\begin{ex}
In the more general setting of \Cref{def:overunder} one does not obtain a satisfactory duality theory in the sense that the operations $\overrightarrow{\ua}^\b$,  $\overleftarrow{\ua}^\b$ need not be mutually inverse even when both sequences $\ua, \ub$ are nondecreasing. Indeed, consider $\a_n=\lceil \frac{n}{2} \rceil$ and $\b_n=\lfloor \frac{n}{2} \rfloor$ which yield $\ua\neq {\overleftarrow{\uau^\b}}^\b$ and $\ua\neq {\overleftarrow{\uau^\b}}^\a$ according to the table below 
\begin{center}
\begin{tabular}{c|cccccc}
n & 1& 2 & 3 & 4 & 5   \\
\hline
$\a_n$ & 1 & 1 & 2 & 2 & 3 \\
$\b_n$ & 0 & 1 & 1 & 2 & 2 \\
$\au^\b_n$ & $-\infty$ & 2 & 2 & 4 & 4 \\
${\overleftarrow{\au^\b}}^\b_n$ & 2 & 2 & 2 & 2 & 2\\
${\overleftarrow{\uau^\b}}^\a_n$ & 2 & 2& 2& 2& 4.
\end{tabular}
\end{center}
\end{ex}

We conclude by considering the transfer of the subadditive and superaditive properties from a sequence to its subsequences.

\begin{lem}
\label{lem:subseq}
Let $\ua=\{\a_n\}_{n\in\N}$ be a sequence. Define for $k\in\Z$ the subsequence $\ua[k]=\{\a_{n+k}\}_{n\in\N,n> -k}$, that is, the $n$-th member of the sequence $\ua[k]$ is $\a_{n+k}$, provided $n+k>0$.
\begin{enumerate}
\item If $k\ge 0$ and $\ua$ is subadditive and nondecreasing, then $\ua[k]$ is subadditive.
\item If $k\le 0$ and $\ua$ is superadditive and nondecreasing, then $\ua[k]$ is superadditive.
\item If $\wa$ exists then $\widehat{\a[k]}$ exists as well and $\widehat{\a[k]}=\wa$.
\end{enumerate}

\end{lem}
\begin{proof}
We focus on assertion (1), the second numbered assertion being similar. Under the hypotheses of (1)
\[
\a[k]_{a+b}=\a_{a+b+k} \leq \a_{a+b+2k} \leq \a_{a+k}+\a_{b+k}=\a[k]_a+\a[k]_b
\]
follows from the non decreasing property of $\ua$  for the first inequality and subbadditivity of $\au$ for the second.
Part (3) follows from
\[
\widehat{\a[k]}=\lim_{n\to \infty}\frac{a_{n+k}}{n}=\lim_{n\to \infty} \frac{a_{n+k}}{n+k}\cdot \lim_{n\to \infty}\frac{n+k}{n}=\wa
\]
\end{proof}

\section{Subadditive and superaditive sequences from graded families}
\label{sec:SequencesFromGradedFamilies}

We are interested in subadditive and superadditive sequences which occur in algebraic contexts. The following considerations introduce types of sequences we shall focus our attention on in the remainder of the manuscript.

\subsection{Valuations and initial degree}
\label{sect:initdeg}

Recall that a discrete valuation on a field $\mathbf{K}$ is a homomorphism $\nu:\mathbf{K}^*\to\Z$ on the units of $\mathbf{K}$ satisfying $\nu(xy) = \nu(x) + \nu(y)$ and $\nu(x + y) \geq \min\{\nu(x), \nu(y)\}$.  If $\mathbf{K}$ is the fraction field of a domain $R$ then a valuation is determined by its values on $R$ via $\nu(f/g)=\nu(f)-\nu(g)$, so we abuse notation by referring to valuations on $R$ instead of its field of fractions.  We furthermore restrict ourselves to valuations which are non-negative on $R$, which we call $R$-valuations.  
\begin{ex}
\label{ex:am}
Given a maximal ideal $\mathfrak{m}$ in a regular ring $R$, a simple example of an $R$-valuation is 
$\alpha_{\mathfrak{m}}(f)=\max\{k:f\in\mathfrak{m}^k\}$.  If $\fm$ is not a maximal ideal, $\alpha_{\fm}$ need not be a valuation; $\alpha_{\fm}(xy)\le \alpha_{\fm}(x)+\alpha_{\fm}(y)$ is always true but equality may not hold~\cite[Section~6.7]{HS06}.
\end{ex}
\begin{defn}
If $\nu$ is an $R$-valuation, denote  the minimum value taken by $\nu$ on $I$ by 
\[
\nu(I)=\min\{\nu(f) : 0\neq f\in I\}
\]
If $R$ is a standard graded ring with homogeneous maximal ideal $\mathfrak{m}$, then the {\bf initial degree} of $I$ is the minimum value taken by the valuation $\a_\m$ in \Cref{ex:am} on $I$
  \[
  \alpha(I)=\min\{\deg{f}: 0\neq f\in I\}=\max\{k:I\subseteq \mathfrak{m}^k\}.
  \]
\end{defn}

Recall that a {\bf graded family} of ideals $\mathcal{I}=\{I_n\}_{n\geq 1}$ of a ring $R$ is a family which satisfies $I_aI_b\subset I_{a+b}$ for all $a,b\in\N$.

\begin{lem}\label{lem:WaldSub}
Given a graded family $\mathcal{I}=\{I_n\}_{n\geq 1}$ of  ideals of a domain $R$ and an $R$-valuation $\nu$ the sequence  $\alpha_n=\nu(I_n)$ is subadditive.
\end{lem}
\begin{proof}
The property $\nu(xy) = \nu(x) + \nu(y)$  implies that $\nu(I_aI_b)=\nu(I_a)+\nu(I_b)$ for all $a,b\in \N$.
It follows from the containment $I_aI_b\subset I_{a+b}$ for all $a,b\in\N$ that $\alpha_{a+b}=\nu(I_{a+b})\leq \nu(I_aI_b)=\nu(I_a)+\nu(I_b)=\alpha_a+\alpha_b$.
\end{proof}

One of the graded families of interest for this paper is formed by symbolic powers.

 \begin{defn}
  Given an ideal $I$ of a ring $R$, 
  the $n^{\mbox{\footnotesize{th}}}$ \textbf{symbolic power} of $I$ is 
  \[
   I^{(n)}=\bigcap_{P\in \Ass(I)} \left(I^nR_{P}\cap R\right).
  \]  We set $I^{(0)}=R$ by convention.
  \end{defn}
  
The growth of initial degree of the symbolic powers of an ideal is captured by the Waldschmidt constant. This invariant, formally introduced in \cite{BH10}, has often been featured implicitly in the geometric literature; see section \ref{s:Nagata} for further details and connections. More generally, the asymptotic growth factor of an arbitrary valuation applied to the symbolic powers of an ideal is  dubbed a skew Waldschmidt constant in~\cite{DFMS}.
  
 \begin{defn}
 \label{def:waldschmidt}
 The {\bf Waldschmidt constant} of a homogeneous ideal $I$ is the real number
 \[
 \widehat{\alpha}(I)=\lim_{n\to \infty}\frac{\alpha(I^{(n)})}{n}=\inf_{n\in\N} \frac{\alpha(I^{(n)})}{n}.
 \]
 Given a valuation $\nu:R\to \Z$, the {\bf skew Waldschmidt constant} of a homogeneous ideal $I$ is the real number
	\[
	\widehat{\nu}(I)=\lim_{n\to \infty}\frac{\nu(I^{(n)})}{n}=\inf_{n\in\N} \frac{\nu(I^{(n)})}{n}.
	\]
 \end{defn}
 
Throughout the paper $\overline{J}$ denotes the integral closure of an ideal $J$. The valuative criterion for integral closures~\cite[Theorem~6.8.3]{HS06} states that for a fixed ideal $I\subset R$ and every $f\in R$, $f\in \overline{I}$ if and only if $\nu(f)\ge \nu(I)$ for every $R$-valuation $\nu:R\to\Z$.  From this we get the following ideal membership test: there is containment $J\subset \overline{I}$ between two ideals if and only if $\nu(J)\ge \nu(I)$ for every $R$-valuation $\nu:R\to\Z$. 
We now define a sequence inspired by this criterion and its applications to the containment problem between the ordinary and symbolic powers of an ideal (see \cref{s:resurgence}). 

We say that a valuation $\nu$ is \textit{supported} on an ideal $I$ if $\nu(I)>0$.  
\begin{defn}
\label{def:betanu}
Given an ideal $I\subset R$ and an $R$-valuation $\nu$ supported on $I$, define
\[
\beta^{\nu}_n=\beta^{\nu}_n(I)=\sup\{d:\nu(I^{(d)}) < \nu(I^n)\}.
\]
\end{defn}

\begin{rem}
The fact that if $R$ is Noetherian $\beta^{\nu}_n\in\N$ for each $n\in\N$ follows from Swanson's theorem on linear equivalence of the symbolic and ordinary $I$-adic topologies. In detail, it is shown in  \cite{Swanson}  that there exists an integer $\ell$ (possibly dependent upon $I$) such that $I^{(\ell n)}\subseteq I^n$ for all $n\in \N$. This yields 
\[
\nu(I^{(\ell n)})\geq \nu(I^n) =n\nu(I) \geq \nu(I)=\nu(I^{(1)}).
\]
Consequently, since the sequence $\{\nu(I^{(d)})\}_{d\in\N}$ is nondecreasing, we have $1\leq \beta^{\nu}_n< \ell n$.
\end{rem}


We come to our first application of \Cref{thm:relativeduality}. 

\begin{prop}\label{lem:valuationLimit}
For any Noetherian domain $R$, any ideal $I\subset R$ and any $R$-valuation $\nu$ supported on $I$ the sequence $\beta^{\nu}_n=\beta^{\nu}_n(I)$ is superadditive and satisfies
\[
\widehat{\b^\nu}=\lim_{n\to\infty}\frac{\beta^\nu_n}{n}=\sup_{n\in\N}\left\lbrace \frac{\beta_n^{\nu}}{n}\right\rbrace=\frac{\nu(I)}{\widehat{\nu}(I)}
\]
\end{prop}
\begin{proof}
We first give an alternate definition for $\beta^\nu$. Set $\gamma_d=\nu(I^{(d)})$ and $\delta_n=\nu(I^n)-1=n\nu(I)-1$ for $n,d\in \N$. Note that $\underline{\gamma}$ is subadditive by \Cref{lem:WaldSub}, $\underline{\delta}$ is superadditive by its definition, and we have $\widehat{\gamma}=\widehat{\nu}(I)$ and $\widehat{\delta}=\nu(I)$. Then \Cref{def:betanu} can be rewritten as $\beta^\nu=\overrightarrow{\gamma^\delta}$.  An application of \Cref{thm:relativeduality} (1) yields that the the sequence $\beta^{\nu}$ is superadditive and $\widehat{\b^\nu}=\nu(I)/\widehat{\nu}(I)$. The first equality  in the claim follows from superadditivity of $\beta^{\nu}$ and Fekete's lemma.  
\end{proof}

In the next subsection we interpret the asymptotic growth factor $\widehat{\b^\nu}$ in terms of an invariant of $I$ termed asymptotic resurgence.

\subsection{Asymptotic resurgence}
\label{s:resurgence}
The various invariants defined below under the name of resurgence were introduced  to study the {\em containment problem} which asks for pairs of natural numbers $n,d$ for which $I^{(n)}\subseteq I^d$.

\label{sec:resurgence}
\begin{defn}
The {\bf resurgence} of an ideal $I$, introduced in~\cite{BH10}, is the quantity
\[
\rho(I)=\sup\left\{\frac{n}{d} : I^{(n)}\not\subseteq I^d \right\}.
\]
Its asymptotic counterpart is the {\bf asymptotic resurgence} of $I$, introduced in~\cite{GHV13}
\[
\wrho(I)=\sup\left\{\frac{n}{d} : I^{(nt)}\not\subseteq I^{dt} \text{ for } t\gg 0 \right\}.
\]
Version of these invariants using integral closures were defined in~\cite{DFMS}.  These are the {\bf ic-resurgence}
\[
\rho_{ic}(I)=\sup\left\{\frac{n}{d} : I^{(n)}\not\subseteq \overline{I^d} \right\}
\]
and the {\bf ic-asymptotic resurgence}
\[
\wrho_{ic}(I)=\sup\left\{\frac{n}{d} : I^{(tn)}\not\subseteq \overline{I^{td}} \text{ for } t\gg 0 \right\}.
\]
  It is shown in~\cite[Corollary~4.14]{DFMS} that $\rho_{ic}(I)=\wrho_{ic}(I)=\wrho(I)$.  By contrast, in general we have $\wrho(I)\neq \rho(I)$; see \cite{DHNST15}.
\end{defn}

In this section we discuss two numerical sequences which  arise in conjunction with these notions of resurgence:
\begin{eqnarray*}
\lambda_n = \lambda_n(I) = \max\{d:I^{(d)}\not\subseteq I^n\}  \text{ and }
\beta_n =\beta_n(I) =\max\{d:I^{(d)}\not\subseteq \overline{I^n}\}.
\end{eqnarray*}
Notice that
\[
\rho(I)=\sup_{n\in\N}\left\lbrace \frac{\lambda_n}{n} \right\rbrace \mbox{ and } \wrho(I)=\sup_{n\in\N} \left\lbrace \frac{\beta_n}{n} \right\rbrace
\]
follows from the definition of resurgence and asymptotic resurgence, respectively.  If $R$ is a regular ring and $I$ is radical then~\cite[Remark~5.5]{DD21} implies that in fact
\[
\lim_{n\to\infty}\frac{\lambda_n}{n}=\lim_{n\to\infty}\frac{\beta_n}{n}=\wrho(I).
\]
The assumption that $I$ is radical can be removed (see~\cite[Remark~4.23]{DD21}).  Thus we see that the sequence $\{\beta_n\}$ behaves like a superadditive sequence in the sense that
\[
\lim_{n\to\infty}\frac{\beta_n}{n}=\sup_{n\in\N}\left\lbrace \frac{\beta_n}{n}\right\rbrace.
\]
Since there are examples where $\rho(I)\neq \wrho(I)$ (see~\cite{DHNST15}), $\{\lambda_n\}$ is not necessarily superadditive.  We do not know if $\ub=\{\beta_n\}_{n\in\N}$ is always a superadditive sequence.  However, we are able to replace $\ub$ by a valuative sequence of the type discussed in \Cref{def:betanu} which is superadditive and whose asymptotic growth rate is also equal to the asymptotic resurgence.

\begin{prop}
\label{prop:betahat=rhohat}
Let $I$ be an ideal in a regular ring $R$.  For any valuation $\nu:R\to\Z$, we have $\beta^\nu_n\le \beta_n\le \lambda_n$.  Moreover there is a choice of valuation $\nu$ so that
\[
\widehat{\b^\nu}=\lim_{n\to\infty} \frac{\beta^\nu_n}{n}=\lim_{n\to\infty} \frac{\beta_n}{n}=\lim_{n\to\infty}\frac{\lambda_n}{n}=\wrho(I).
\]
\end{prop}
\begin{proof}
The inequalities $\beta^\nu_n\le \beta_n\le \lambda_n$ follow from the definitions of the sequences and the valuative criterion for integral closures.  The equalities
\[
\lim_{n\to\infty} \frac{\beta_n}{n}=\lim_{n\to\infty}\frac{\lambda_n}{n}=\wrho(I)
\]
follow from~\cite[Remark~5.5]{DD21}, as noted above.  By~\cite[Theorem~4.10]{DFMS}, $\wrho(I)=\nu(I)/\widehat{\nu}(I)$ for some choice of valuation (in fact, one of the \textit{Rees} valuations of $I$ will accomplish this).  Hence, for this valuation, \Cref{lem:valuationLimit} yields that $\lim_{n\to\infty} \beta^\nu_n/n=\wrho(I)$, completing the proof.
\end{proof}

\subsection{Castelnuovo-Mumford regularity}
\begin{defn}
\label{def:reg}
Suppose $R=\bigoplus_{i\ge 0} R_i$ is a graded ring with residue field $\kk=R/R_{+}$ where $R_{+}=\bigoplus_{i>0} R_i$.
The {\bf Castelnuovo-Mumford  regularity} of a graded module $M$ over $R$ is
\[
\reg(M)=\max\{j-i \mid {\rm Tor}^R_i(M,\kk)_j\neq 0\}.
\]
 When $M$ has finite length and $R$ is standard graded the regularity  can also be expressed as $\reg(M)=\soc(M):=\max\{i \mid M_i\neq 0\}$. Moreover there is an alternate definition in terms of the local cohomology modules of $M$ supported at the homogeneous maximal ideal $\fm$
 \[
\reg(M)=\sup\{\soc\left(H^i_\fm(M)\right)+i \mid 0\leq i\leq \dim(M)\}.
\]
\end{defn}

Keeping with the theme of our writing, we are interested in families of ideals or modules whose Castelnuovo-Mumford regularities give subadditive or superadditive sequences. The subbaditive case is considered in the following lemma, while a family with superadditive regularity sequence is illustrated in \Cref{rem:betasuperadditive}.
\begin{lem}
\label{lem:CM}
Given a graded family $\mathcal{I}=\{I_n\}_{n\in \N}$ of homogeneous ideals of a standard graded ring $R$ so that each quotient ring $R/I_n$ is Cohen-Macaulay of the same dimension $\dim(R/I_n)=d$. Then the sequences of Castelnuovo-Mumford regularities of the members in the family
$
\{\reg(I_n)\}_{n\in \N}
$
is subadditive.
\end{lem}
\begin{proof}
We may assume that the residue field of $R$ is infinite by tensoring with an infinite extension of the base field if necessary; this does not change the Castelnuovo-Mumford regularity of the given ideals. Thanks to the Cohen-Macaulay property one may reduce to the Artinian case. In detail, fix $a,b\in \N$ and choose a sequence of linear forms $\ell_1, \ldots, \ell_{d}$ which is simultaneously a regular sequence on $R/I_a, R/I_b$ and also on $R/I_{a+b}$. Now set $\widetilde{R}=R/(\ell_1, \ldots, \ell_d)$ and $\widetilde{I}_n=I_n+(\ell_1, \ldots, \ell_d)/(\ell_1, \ldots, \ell_d)$ for $n\in\{a,b,a+b\}$. This gives that $\reg(\widetilde{I}_n)=\reg(I_n)$ for $n\in\{a,b,a+b\}$. Moreover, setting $\fm$ to be the homogeneous maximal ideal of $\widetilde{R}$, since each of the quotients $\widetilde{R}/\widetilde{I}_n$ is Artinian we have for $n\in\{a,b,a+b\}$
\[
r_n=\reg(I_n)=\reg\left(\widetilde{I}_n\right)=\min\{d: \left(\widetilde{R}/\widetilde{I}_n\right)_d=0\}=\min\{d: \fm^d\subseteq \widetilde{I}_n\}.
\]

It follows from the containment $I_aI_b\subset I_{a+b}$  that $\widetilde{I}_a\widetilde{I}_b\subset \widetilde{I}_{a+b}$. 
We deduce
\[
\fm^{r_a+r_b}=\fm^{r_a}\fm^{r_b}\subseteq \widetilde{I}_a\widetilde{I}_b\subseteq \widetilde{I}_{a+b}
\]
and thus it follows that $r_{a+b}=\reg\left(\widetilde{I}_{a+b}\right)\leq r_a+r_b$.
\end{proof}

\begin{defn}
\label{def:aspCM}
Ideals $I$ for which every member of the sequence of symbolic powers  $\{I^{(n)}\}$ yields a Cohen-Macaulay quotient are dubbed {\bf aspCM} ideals in \cite{walker}. 
\end{defn}
The aspCM class includes complete intersection ideals, saturated ideals with $\dim R/I=1$, that is defining ideals for finite sets of points or fat points (not necessarily reduced schemes supported at finite sets of points) in $\P^N$, ideals defining matroid configurations in $\P^N$ \cite{GHMN}, and generic determinantal ideals.

\begin{rem}
Even under the hypotheses of \Cref{lem:CM}, the closely related sequence $\{\reg(R/I_n)\}_{n\in\N}$ need not be subadditive. Take for example $I_n=(f^n)$ where $f$ is a homogeneous element of degree $d>0$ in a standard graded polynomial ring $R$. Then $\reg(R/I_n)=dn-1$ is not subadditive.
\end{rem}

We define an  invariant which captures the asymptotic growth of the regularity for a family of ideals.
  
 \begin{defn}
 \label{def:areg}
 The {\bf asymptotic regularity} of a family $\cI=\{I_n\}_{n\in \N}$ of homogeneous ideals  is the following limit, provided it exists,
 \[
 \widehat{\reg}(\cI)=\lim_{n\to \infty}\frac{\reg(I_n)}{n}.
 \]
 \end{defn}
 
By way of Fekete's lemma, \Cref{lem:CM} provides a set of assumptions under which the limit in \Cref{def:areg} exists. We shall be primarily interested in the asymptotic regularity for the family of symbolic powers $\cI=\{I^{(n)}\}_{n\in \N}$ of a given ideal $I$, which we denote $ \widehat{\reg}(I)$. This family does not always satisfy the conditions of  \Cref{lem:CM}, but it does so, for example, when $I$ defines a finite set of points in projective space. In this case, the existence of  $ \widehat{\reg}(I)$ also follows from the much more general result in \cite[Theorem B]{CEL01}. There are few other instances where the existence of $\widehat{\reg}(I)$ is known, for example, when $I$ is a monomial ideal cf. \cite[Theorem 3.6.]{DHNT21}, or more generally, when the symbolic Rees algebra of $I$ is Noetherian, which is shown in the ongoing work of the second author with Hop and H\`a. 

The next example points out that in general the sequence $\{\reg(I_n)\}_{n\in\N}$ need not be subadditive for a graded family of ideals $\mathcal{I}=\{I_n\}_{n\in\N}$ even when that family consists of symbolic powers of monomial ideals and thus $\widehat{\reg}(\cI)$ exists. The ideals $J(m,s)$ in the next example yields Cohen-Macaulay quotient rings, but their symbolic powers do not, thus they are not aspCM.
 
 \begin{ex}
 \label{ex:nonsubreg}
 In \cite[Theorem 5.15]{DHNT21} Dung, Hien, Nguyen, and Trung produce examples of squarefree monomial ideals $J(m,s)$ such that 
 \[
 \reg \left(J(m,s)^{(t)}\right)=
 \begin{cases}
 m(s + 1)n & t=2n\\
 m(s + 1)n + m + s - 1 &t=2n+1
 \end{cases}.
 \]
 Combinatorially the ideals $J(m,s)$ are described as cover ideals for corona graphs obtained by adding $s$ pendant edges to each vertex of a complete graph $K_m$. The ideals in this family were singled out as examples of squarefree monomial ideals for which the function $t\mapsto \reg(J(m,s)^{(t)})$ is not eventually linear. For these symbolic power ideals the regularity matches the largest degree of a minimal generator, which shows that if an ideal $J$ is generated in degrees $\leq d$ one cannot conclude that $J^{(t)}$ is generated in degrees $\leq td$. This relates to a question of Huneke \cite[Problem 0.4]{HunekeAIM}.

 A necessary condition for the sequence $\{\reg(J(m,s)^{(t)})\}_{t\in\N}$ to be subbaditive is
 \[
\reg(J(m,s)^{(2t_1+2t_2+2)}) \leq\reg(J(m,s)^{(2t_1+1)})+ \reg(J(m,s)^{(2t_2+1)})
 \]
 which can be written equivalently as
 \[
 m(s+1)\leq 2m+2s-2 \text{ or } (m-2)(s-1)\leq 0.
 \]
 It is thus evident that the regularity sequence for the symbolic powers of $J(m,s)$ is not subadditive whenever $m\geq 2$ and $s\geq 1$ but $(m,s)\neq(2,1)$. 
 \end{ex}

\section{Inverse systems of differentially closed graded filtrations}\label{sec:InverseSystems}

One situation in which the duality between subadditive sequences and superadditive sequences naturally arises is in the theory of inverse systems.  In this section we extend a construction using inverse systems from an influential paper of Emsalem and Iarrabino~\cite{EI95}; two sequences naturally associated to this construction exhibit the duality of Section~\ref{s:sequences}.  We begin by recalling some details about contraction and differentiation, following the survey of Geramita~\cite[Lecture~9]{Ger95}.

\subsection{Contraction, differentiation, and inverse systems}
Let $\kk$ be a field and $R=\kk[x_0,\ldots,x_N]$.  We use a standard shorthand for monomials -- if $\mathbf{a}=(a_0,\ldots,a_N)\in \Z^{N+1}_{\ge 0}$, then $x^\mathbf{a}=x_0^{a_0}\cdots x_N^{a_N}$ is the corresponding monomial in $R$.  We define $\mathcal{D}=\bigoplus_{i\geq 0} {\rm Hom}(R_i,\kk)$, the graded $\kk$-dual of $R$.  If $x^\mathbf{a}$ is in $R_d$, we write $Y^{[\mathbf{a}]}$ for the functional (in $\mathcal{D}_d$) on $R_d$ which sends $x^{\mathbf{a}}$ to $1$ and all other monomials in $R_d$ to $0$.  As a vector space, $\mathcal{D}$ is isomorphic to a polynomial ring in $N+1$ variables.  However, as we recall shortly, $\mathcal{D}$ has the structure of a \textit{divided power algebra}.  For this reason, we call $Y^{[\mathbf{a}]}$ a \textit{divided} monomial.

The ring $R$ acts on $\mathcal{D}$ by \textit{contraction}, which we denote by $\contract$.  That is, if $x^{\mathbf{a}}$ is a monomial in $R$ and $Y^{[\mathbf{b}]}$ is a divided monomial in $\mathcal{D}$, then
\[
x^{\mathbf{a}}\contract Y^\mathbf{b}=Y^{[\mathbf{b}-\mathbf{a}]} \text{ if } \mathbf{b} \ge \mathbf{a},
\]
and $0$ otherwise.  This action is extended linearly to all of $R$ and $\mathcal{D}$.  This action of $R$ on $\mathcal{D}$ gives a perfect pairing of vector spaces $R_d\times \mathcal{D}_d \to \kk$ for any degree $d\ge 0$.  Suppose $U$ is a subspace of $R_d$.  We define
\[
U^{\perp}=\{g\in \mathcal{D}_d: f\contract g=0 \mbox{ for all } f\in U\}.
\]
Macaulay \cite{Macaulay} introduced the {\em inverse system} of an ideal $I$ of $R$ to be
\[
I^{-1}:= \mbox{Ann}_S(I)=\{g\in \mathcal{D}: f\contract g=0 \mbox{ for all } f\in I\}.
\]
If $I$ is a homogeneous ideal of $R$ then the inverse system $I^{-1}$ can be constructed degree by degree using the identification $(I^{-1})_d=I_d^{\perp}$~\cite[Proposition~2.5]{Ger95}.  In general, $I^{-1}$ is an $R$-submodule of $\mathcal{D}$ which is finitely generated if and only if $I$ is an Artinian ideal.

A priori, $\mathcal{D}$ is simply a graded $R$-module.  However, $\mathcal{D}$ can be equipped with a multiplication which makes it into a ring, called the divided power algebra.  Suppose $\mathbf{a}=(a_0,\ldots,a_N),\mathbf{b}=(b_0,\ldots,b_N)\in \Z^{N+1}_{\ge 0}$.  The multiplication in $\mathcal{D}$ is defined on monomials by
\begin{equation}\label{eq:dividedmultiplication}
Y^{[\mathbf{a}]}Y^{[\mathbf{b}]}= {\mathbf{a}+\mathbf{b} \choose \mathbf{a} } Y^{[\mathbf{a+b}]},
\end{equation}
where
\begin{equation}\label{eq:multifactorial}
\mathbf{a}!=\prod_{i=0}^N a_i \quad\mbox{and}\quad \binom{\mathbf{a}+\mathbf{b}}{\mathbf{a}}=\prod_{i=0}^N \binom{a_i+b_i}{a_i}.
\end{equation}
This multiplication is extended linearly to all of $\mathcal{D}$.  Let $\mathbf{e}_i$ be the $i$th standard basis vector in $\Z^{N+1}$ and put $Y_i:=Y^{\mathbf{e_i}}$.  We see from the above definition that if $\mathbf{a}=(a_0,\ldots,a_N)$ then $Y^{[\mathbf{a}]}=\prod_{i=0}^N Y_i^{[a_i]}$.  Now set $Y^{\mathbf{a}}=\prod_{i=0}^N Y_i^{a_i}$, where the multiplication occurs in the divided power algebra as defined above.  From the above definition, one can deduce that
\begin{equation}\label{eq:regularmonomialtodividedmonomial}
Y^{\mathbf{a}}=\mathbf{a}! Y^{[\mathbf{a}]}.
\end{equation}
See~\cite[Lecture~9]{Ger95} for additional details.  In characteristic zero, $\mathbf{a}!$ never vanishes and so $\mathcal{D}$ is generated as an algebra by $Y_0,\ldots,Y_N$, just like the polynomial ring.  However, in charateristic $p$, $\mathcal{D}$ is infinitely generated by all the divided power monomials $Y_j^{[p^{k_i}]}$ for all $j=0,\ldots ,N$ and $k_j\ge 0$. The proof of this fact follows from Lucas' identity: given base $p$ expansions $a=\sum a_ip^i$ and $b=\sum b_ip^i$ for $a, b\in\N$,  then
\[
{b \choose a} = \prod_{i=0}^\infty {b_i \choose a_i} \mod p.
\]
In particular, for any $\mathbf{a}=(a_0,\ldots,a_N)$ where $a_j=\sum a_{ij}p^i$, we have
\[
Y^{[\mathbf{a}]}=\prod_{j=0}^N \prod_{i} (Y_j^{[p^i]})^{a_{ij}}.
\]

We now revisit the characteristic zero case.  Suppose $\kk$ is a field of characteristic zero and let $S=\kk[y_0,\ldots,y_N]$ be a polynomial ring.  Consider the action of $R$ on $S$ by partial differentiation, which we represent by `$\circ$'.  That is, if $\mathbf{a}=(a_0,\ldots,a_N)\in \Z^{N+1}_{\ge 0}$, $x^\mathbf{a}=x_0^{a_0}\cdots x_N^{a_N}$ is a monomial in $R$, and $g\in S$, we write
\[
x^{\mathbf{a}}\circ g=\frac{\partial^{\mathbf{a}}g}{\partial x^{\mathbf{a}}}
\]
for the action of $x^{\mathbf{a}}$ on $g$ (extended linearly to all of $R$).  In particular, if $\mathbf{a}\le \mathbf{b}$, then
\[
x^{\mathbf{a}}\circ y^{\mathbf{b}}=\frac{\mathbf{b}!}{(\mathbf{b}-\mathbf{a})!}y^{\mathbf{b}-\mathbf{a}},
\]
where we  use~\eqref{eq:multifactorial}.  This action gives a perfect pairing $R_d\times S_d\to \kk$, and, given a homogeneous ideal $I\subset R$, we define $I_d^{\perp}$ and $I^{-1}$ in the same way as we do for contraction.

Since we are in characteristic zero, the map of rings $\Phi:S\to\mathcal{D}$ defined by $\Phi(y_i)=Y_i$ extends to all monomials via~\eqref{eq:regularmonomialtodividedmonomial} to give $\Phi(y^\mathbf{a})=Y^{\mathbf{a}}=\mathbf{a}! Y^{[\mathbf{a}]}$.  Thus $S$ and $\mathcal{D}$ are isomorphic.  Moreover, if $F\in R$ and $g\in S$, then $\Phi(F\circ g)=F\contract\Phi(g)$~\cite[Theorem~9.5]{Ger95}, so $S$ and $\mathcal{D}$ are isomorphic as $R$-modules.

\subsection{Differential operators and differentially closed filtrations}
If $\ba=(a_0,\ldots,a_N)\in\Z^{N+1}_{\ge 0}$, we extend our convention on monomials to differential operators,  letting $\frac{\partial^{\ba}}{\partial x^\ba}=\frac{\partial^{a_0}}{\partial x_0^{a_0}} \cdots \frac{\partial^{a_N}}{\partial x_N^{a_N}}$.  Independent of characteristic, the ring of $\kk$-linear differential operators $D_R$, which acts on $R$, can be written as $D_R=\cup_{n\in \mathbb{N}}D^n_R$, where
\[
D^n_R=R \left< \frac{1}{\ba!} \frac{\partial^\ba}{\partial x^{\ba}} ~\Big|~ |\ba| \leq n \right>,
\]
where we use the convention~\eqref{eq:multifactorial} for $\ba!$.  See~\cite[Remark~2.7]{DDSGHN2018}.  For simplicity, we will write $D_{\ba} = \frac{1}{\ba!} \frac{\partial^{\ba}}{\partial x^\ba}$. The factors $\frac{1}{a_i!}$ appearing in $\frac{1}{\ba!}$ do not represent elements in the field; $D_{\ba}$ is a formal representation for the $\kk$-linear operator defined by $D_{\ba}(x^\bb) = {\bb \choose \ba} x^{\bb-\ba}$ if $\bb\geq \ba$, and otherwise $D_{\ba}(x^\bb) = 0$. Note that in characteristic $0$, $\ba! D_{\ba}=  \frac{\partial^{\ba}}{\partial x^{\ba}}$ is the usual partial differential operator.  Thus in characteristic $0$, $D_R$ is generated as an $R$-algebra by either $D_{\be_i}$ for $i=0,\ldots,N$ (where $\be_i$ is the $i$th standard basis vector) or by $\frac{\partial}{\partial x_i}$ for $i=0,\ldots,N$.

In characteristic $p$, using Lucas' identity as in the divided power ring, one can show that if $\ba = (a_0,\ldots ,a_N)$ where $a_j = \sum a_{ij}p^i$, then $D_\ba = \prod_{j=0}^N \prod_i (D_{p^i\mathbf{e_j}})^{a _{ij}}$ where the product just means the composition of the operators.  This computation shows that $D_R$ is generated as an $R$-algebra by $\{D_{p^i\mathbf{e_j}}: 0\leq j\leq N, 0\le i\}$.

\begin{defn}\label{def:differentiallyclosedfiltration}
Suppose $R=\kk[x_0,\ldots,x_N]$ and let $\I=\{I_n\}_{n\ge 1}$ be a filtration of ideals.  We say that $\I$ is \textbf{differentially closed} if, for every $n\ge 0$, every $D_\ba\in D^{n-1}_R$, and every $F\in I_{n}$, $D_{\ba}F\in I_{n-|\ba|}$.
\end{defn}

The following two lemmas follow immediately from our discussion of the $R$-algebra generators of $D_R$.

\begin{lem}\label{lem:DiffClosedChar0}
Suppose $R=\kk[x_0,\ldots,x_N]$, where $\kk$ has characteristic zero, and $\I=\{I_n\}_{n\ge 1}$ is a filtration of ideals so that for every $n\ge 1$ and every $F\in I_{n+1}$, $\frac{\partial F}{\partial x_i}\in I_n$ for $i=0,\ldots,N$.  Then $\I$ is differentially closed.
\end{lem}

\begin{lem}\label{lem:p^iclosedisenough}\label{lem:DiffClosedCharP}
Suppose $R=\kk[x_0,\ldots,x_N]$, where $\kk$ has characteristic $p>0$, and $\I=\{I_n\}_{n\ge 1}$ is a filtration of ideals so that for every $i\in \N$, $n\ge 1+p^i$ and every $F\in I_{n}$, $D_{p^i\mathbf{e_j}}F\in I_{n-p^i}$ for $j=0,\ldots,N$, where $\mathbf{e_j}$ is the $j$ th standard basis vector of $\Z^{N+1}$.  Then $\I$ is differentially closed.
\end{lem}


\begin{ex}\label{ex:SymbolicPowersDiffClosed}
Let $I\subset R=\kk[x_0,\ldots,x_N]$ be a homogeneous ideal.  The $n$th \textit{differential power} of $I$ is
\[
I^{<n>}=\{f\in R\mid D_\ba (f)\in I \mbox{ for all } D_\ba\in D^{n-1}_R\}.
\]
Every differential power of $I$ is an ideal by~\cite[Proposition~2.4]{DDSGHN2018}.  The family $\cI=\{I^{<n>}\}_{n=1}^\infty$ is clearly a differentially closed graded filtration of ideals.
	
If $I=I(X)$ is the ideal of a projective variety $X\subset\P^N$ for $\kk$ characteristic $0$ or a radical ideal for $\kk$ a perfect field, the Zariski-Nagata theorem~\cite{Zariski-1949,Nagata59} and its extension to perfect fields~\cite[Proposition~2.14]{DDSGHN2018} states that the symbolic powers and differential powers of $I$ coincide, that is $I^{(n)}=I^{<n>}$ for $n\ge 1$.  In either case, $\I=\{I^{(n)}\}_{n\ge 1}$ is a differentially closed graded filtration.  We will see in \Cref{ex:ZariskiMainLemmaHolomorphic} that, by using Zariski's main lemma on holomorphic functions~\cite{Zariski-1949} instead of the Zariski-Nagata theorem, we can drop the assumption that $\kk$ is perfect.
\end{ex}

\begin{ex}\label{ex:PowersDiffClosed}
Suppose $R=\kk[x_0,\ldots,x_N]$, $I\subset R$ is any ideal, and $\I=\{I^{n}\}_{n\ge 1}$ is the graded filtration consisting of powers of $I$.  We prove that $\I$ is differentially closed.

In characteristic $0$, $\I$ is differentially closed by \Cref{lem:DiffClosedChar0} and the product rule.
To prove that $\I=\{I^n\}$ is a differentially closed graded filtration in arbitrary characteristic, it suffices by \Cref{lem:DiffClosedCharP} to prove that
\begin{equation}\label{eq:powerscharp}
\mbox{if }f\in I^n \mbox{ then } D_{k\be_i}(f)\in I^{n-k} \mbox{ for } k\le n-1.
\end{equation}
We prove this using the following extension of the product rule for differential  operators of the form $D_{k\be_i}$:  for any $f,g\in R$
\begin{equation}\label{eq:higherproductrule}
D_{k\be_i}(fg)=\sum_{j=0}^k D_{j\be_i}(f)D_{(k-j)\be_i}(g).
\end{equation}
We include a proof of this identity in \Cref{app:formulas}.  From \eqref{eq:higherproductrule} an induction yields
\begin{equation}\label{eq:higherproductrulemany}
D_{k\be_i}(f_1\cdots f_n)=\sum_{j_1+\cdots+j_n=k} D_{j_1\be_i}(f_1)D_{j_2\be_i}\cdots D_{j_n\be_i}(f_n),
\end{equation}
where the sum runs over non-negative integers $j_1,\ldots,j_n$.  To prove~\eqref{eq:powerscharp} it suffices, by linearity, to prove it in the case $f=f_1\cdots f_n$, where $f_i\in I$ for $i=1,\ldots,n$.  Since $k<n$, at least $n-k$ of the indices $j_1,\ldots,j_n$ are zero.  Thus  each term in~\eqref{eq:higherproductrulemany} is a product that includes at least $n-k$ factors in $I$, and so each term is in $I^{n-k}$.  This proves that $\I=\{I^n\}_{n\ge 1}$ is differentially closed in arbitrary characteristic.

\end{ex}

\begin{ex}
\label{ex:FrobeniusPowers}
Suppose $I\subset \kk[x_0,\ldots,x_N]$ is an ideal and $\kk$ has characteristic $p$.  If $q=p^e$ for some integer $e\ge 0$ then the $q$th \textit{Frobenius power} of $I$ is the ideal
\[
I^{[q]}=\langle f^{q}: f\in I\rangle.
\]
In~\cite{Hernandez-Teixera-Witt-2020}, Hern\'{a}ndez, Teixera, and Witt introduce \textit{integral} Frobenius powers
\[
I^{[n]}=I^{n_0}I^{n_1[p]}\cdots I^{n_s[p^s]},
\]
where $n$ has base $p$ expansion $n=n_0+n_1p+\cdots+n_sp^s$ and $I^{a[q]}=(I^a)^{[q]}=(I^{[q]})^a$.  Let $\I=\{I^{[n]}\}$.  We show that $\I$ is a differentially closed filtration.

First we show that if $q=p^t$ (for any integer $t\ge 0$) then differential operators of order not divisible by $q$ vanish on $I^{[q]}$.  By \Cref{lem:DiffClosedCharP} it suffices to show that $D_{k\be_i}(f)=0$ when $q\nmid k$ and $f \in I^{[q]}$.
To this end, suppose that $f=g^{q}$ for $q=p^t$ and $k$ is a positive integer so that $q\nmid k$.  Since $f$ is a linear combination of $q$th powers of monomials, it suffices to show that $D_{k\be_i}(f)=0$ when $f$ is a monomial.  So suppose that $f=x^{q\ba}$, where $\ba=(a_0,\ldots,a_N)$.  Then $D_{k\be_i}(f)=D_{k\be_i}(x^{q\ba})=\binom{q\alpha_i}{k}x^{q\ba-k\be_i}$.  Since $q$ does not divide $k$ the base $p$ expansion $k=\sum_{i\ge 0}k_ip^i$ satisfies $k_u\neq 0$ for some $u<t$.  On the other hand, the base $p$ expansion $qa_i=\sum_{j\ge 0}a_{ij}p^j$ satisfies $a_{ij}=0$ for all $j<t$.  In particular, $a_{iu}=0$.  By Lucas' identity, $\binom{qa_i}{k}=0$ and we are done.  Notice that since $I^{a[q]}=(I^a)^{[q]}$, this also shows that differential operators of order not divisible by $q$ vanish on $I^{a[q]}$ for any $a\ge 1$.

Next we show that, if $1\le k\le a$ and $f\in I^{a[q]}$, then $D_{kq\be_i}(f)\in I^{(a-k)[q]}$ (where we take $I^0=R$ by convention).  It suffices by linearity to consider the case $f=f_1\cdots f_a\in I^{a[q]}$, where $f_1,\ldots,f_a\in I^{[q]}$.  By~\eqref{eq:higherproductrulemany},
\[
D_{k\be_i}(f_1\cdots f_a)=\sum_{j_1+\cdots+j_a=kq} D_{j_1\be_i}(f_1)D_{j_2\be_i}(f_2)\cdots D_{j_a\be_i}(f_a),
\]
and by the previous discussion we may assume that in the sum above $j_1,\ldots,j_a$ are all divisible by $q$.  Hence in each term of the sum above there are $a-k$ factors which are in $I^{[q]}$, thus the entire sum is in $I^{(a-k)[q]}$.

Finally, we induct on the length of the base $p$ exansion of $n$ to show that $D_{j\be_i}(I^{[n]})\subset I^{[n-j]}$ for $j<n$.  If $n<p$ then integral Frobenius powers agree with regular powers and the result follows from Example~\ref{ex:PowersDiffClosed}.  So suppose that $n\ge p$ with base $p$ expansion $n=n_0+n_1p+\cdots+n_sp^s$.  Put $n'=n-n_sp^s$.  Clearly the base $p$ expansion of $n'$ has length at least one less than the base $p$ expansion of $n$.  By definition, $I^{[n]}=I^{[n']}I^{n_s[p^s]}$, so it suffices to show that $D_{j\be_i}(fg)\in I^{[n-j]}$ where $f\in I^{[n']}$ and $g\in I^{n_s[p^s]}$.  Put $q=p^s$ and suppose $j=aq+r$ where $0\le r<q$ (we must have $a\le n_s$ since $j<n$). By~\eqref{eq:higherproductrule},
\[
D_{j\be_i}(fg)=\sum_{m=0}^j D_{m\be_i}(f)D_{(j-m)\be_i}(g)=D_{r\be_i}(f)D_{aq\be_i}(g),
\]
since all differential operators of order not divisible by $q$ vanish on $g$ and all differential operators of order at least $q$ vanish on $f$. By induction, $D_{r\be_i}(f)\in I^{[n'-r]}$.  By the previous argument, $D_{aq\be_i}(g)\in I^{(n_k-a)[q]}$.  Since there are no base $p$ carries in the addition $(n'-r)+(n_s-a)q$, we have $I^{[n'-r]}I^{(n_s-a)[q]}=I^{[n'+n_sq-(aq+r)]}=I^{[n-j]}$ by~\cite[Proposition~3.4]{Hernandez-Teixera-Witt-2020}.  By Lemma~\ref{lem:DiffClosedCharP}, this completes the proof that the integral Frobenius powers of an ideal form a differentially closed filtration.
\end{ex}

\subsection{The inverse system of a differentially closed filtration}

Emsalem and Iarrobino made a remarkable observation in~\cite{EI95}: even though the inverse system of an ideal is not finitely generated, one could put together the graded pieces of the inverse systems of successive symbolic powers of an ideal to get an ideal of $S$ or $\mathcal{D}$, respectively.  We show that this observation of Emsalem and Iarrobino can be extended to a differentially closed graded filtration of ideals, using the following definition.

\begin{defn}\label{def:LTransform}
Suppose that $\I=\{I_n\}_{n\in\N}$ is a filtration of homogeneous ideals.  For each integer $s\ge 1$ we define
\[
\L^s(\cI):=\bigoplus_{d\ge s+1} \left(I_{d-s}^{-1}\right)_d=\bigoplus_{d\ge s+1} \left(I_{d-s}^{\perp}\right)_d.
\]
If the graded filtration $\I$ is understood, we write $\L^s$ instead of $\L^s(\I)$.  If the inverse system is computed using the partial differentiation action of $R$ on $S$, $\L^s(\I)$ is a subspace of $S$, while if the inverse system is computed using the contraction action of $R$ on $\mathcal{D}$, $\L^s(\I)$ is a subspace of $\mathcal{D}$.
\end{defn}

If $\I=\{I_n\}_{n\in\N}$ is a graded family of ideals of $R$, we have defined $\L^s(\I)$ so that
\[
\left(I^{-1}_n\right)_d= \L^{d-n}(\I)_d
\]
and hence
\[
\left(I_n\right)_d\cong\left( \frac{S}{\L^{d-n}(\I)}\right)_d
\]
or equivalently
\begin{equation}
\label{eq:SmodLvsI}
\left( \frac{S}{\L^{s}(\I)}\right)_d\cong (I_{d-s})_d.
\end{equation}

\begin{ex}\label{ex:LsForPointIdeals}
	Suppose $p=[a_0:\cdots:a_N]\in\P^N$, let $\fm_p\subset R=\kk[x_0,\ldots,x_N]$ be the ideal of homogeneous polynomials vanishing at $p$, and put $\I=\{\fm_p^n\}_{n\ge 1}$.  According to~\Cref{ex:PowersDiffClosed} this is a differentially closed graded filtration.
	
	For the action of $R$ on $S$ by partial differentiation, let $L_p=a_0y_0+\ldots+a_Ny_N\in S$ be the dual linear form of the point $p\in X$.  It follows from~\Cref{def:LTransform} and~\Cref{lem:OnePointInverseSystem} that 
	\[
	\cL^s(\I)=\bigoplus_{d\ge 1} (\fm_p^d)^\perp_{d+s}=\bigoplus_{d\ge 1} \langle L_p^{s+1}\rangle_{d+s}=\langle L_p^{s+1} \rangle.
	\]
	For the action of $R$ on $\mathcal{D}$ by contraction, let $L_p=a_0Y_0+\cdots+a_NY_N\in\mathcal{D}$ be the dual linear form of $p$ and put $L^{[k]}_p=\sum_{|\bb|=k} a_0^{b_0}\cdots a_N^{b_N} Y^{[\bb]}$.  It follows from ~\Cref{def:LTransform} and~\Cref{lem:OnePointInverseSystem} that
	\[
	\begin{array}{rl}
		\cL^s(\I) & =\bigoplus_{d\ge 1} (\fm_p^d)^\perp_{d+s}\\[10 pt]
		&=\bigoplus_{d\ge 1} {\rm span}\{ Y^{[\ba]}L_p^{[c]}: s+1\le c \le d+s, |\ba|=d-c \}\\[10 pt]
		&=\langle L_p^{[c]}: c\ge s+1\rangle.
	\end{array}
	\]
	Note that $\cL^s(\I)$ is an ideal of $S$ (respectively $\mathcal{D}$), although it is not a finitely generated ideal of $\mathcal{D}$ in positive characteristic.
\end{ex}

Our main result in this section is that $\cL^s$ is an ideal of $\mathcal{D}$ (or $S$) precisely when $\I$ is a differentially closed filtration of homogeneous ideals.

\begin{thm}\label{thm:differentiallyclosed}
Suppose $R=\kk[x_0,\ldots,x_N]$ and let $\I=\{I_n\}_{n\ge 1}$ be a filtration of homogeneous ideals.  Then $\cL^s(\I)$ is an ideal of $\mathcal{D}$ (arbitrary characteristic) or $S$ (characteristic $0$) if and only if $\I$ is differentially closed.
\end{thm}

In the proof of Theorem~\ref{thm:differentiallyclosed}, we will use the following formula which we expect is known to experts.  We give a proof of this identity (and others) in \Cref{app:formulas}.

\begin{lem}\label{lem:prodrule}
Suppose $F\in R$ is a homogeneous polynomial and $g\in \mathcal{D}$ is a homogeneous divided power polynomial.  In arbitrary characteristic,
\[
F\contract(Y^{[k]}_jg)=\sum_{i=0}^k Y_j^{[k-i]}(D_{i\mathbf{e_j}}(F)\contract g)
\]
for $j=0,\ldots,N$.
\end{lem}

\begin{proof}[Proof of Theorem~\ref{thm:differentiallyclosed}]
We prove the result for $R$ acting on $\mathcal{D}$ by contraction.  Put $\cL^s=\cL^s(\I)$.  Note that $\cL^s$ is an ideal of $\mathcal{D}$ if and only if $Y^{[\bb]}g\in \cL^s$ for every algebra generator $Y^{[\bb]}$ of $\mathcal{D}$.  In any characteristic, $\cL^s$ is an ideal if and only if $Y_j^{[k]}g\in \cL^s$ for every $j=0,\ldots,N$ and any $k\ge 1$ by \Cref{lem:DiffClosedChar0} and \Cref{lem:DiffClosedCharP}.  Since $\cL^s$ is clearly graded, we may assume $g$ is homogeneous, say of degree $d$.  By \Cref{def:LTransform}, $g\in \cL^s_d$ if and only if $g\in \left(I^{-1}_{d-s}\right)_d$.  It follows that $\cL^s$ is an ideal if and only if $Y_j^{[k]}g\in \cL^s_{d+k}=\left(I^{-1}_{d+k-s}\right)_{d+k}$ for all $d\ge s+1$,  $0\le j\le N$, $k\ge 1$, and $g\in \left(I^{-1}_{d-s}\right)_d$.
	
Fix a degree $d$ and an index $0\le j\le N$.  Then $Y_j^{[k]}g\in\left(I^{-1}_{d+k-s}\right)_{d+k}$ for all $g\in\left(I^{-1}_{d-s}\right)_d$ if and only if $F\contract(Y^{[k]}_jg)=0$ for every $F\in \left(I_{d+k-s}\right)_{d+k}$ and $g\in\left(I^{-1}_{d-s}\right)_{d}$.  By Lemma~\ref{lem:prodrule},
\[
\begin{array}{rl}
F\contract(Y_j^{[k]}g) & = D_{k\be_j}(F)\contract g + Y_j^{[1]}(D_{(k-1)\be_j}(F)\contract g)+\cdots+Y_j^{[k]}(F\contract g)\\
& =D_{k\be_j}(F)\contract g,
\end{array}
\]
where the final equality follows because $D_{t\be_j}(F)$ has degree at least $d+1$ for $0\le t\le k-1$ and $g$ has degree $d$.  It follows that $Y_j^{[k]}g\in\left(I^{-1}_{d+k-s}\right)_{d+k}$ if and only if $D_{k\be_j}(F)\contract g=0$ for all $F\in \left(I_{d+k-s}\right)_{d+k}$, which is to say $D_{k\be_j}(F)\in (I_{d-s})_d$  for all $F\in \left(I_{d+k-s}\right)_{d+k}$.  Thus $\I$ is differentially closed if and only if $\cL^s$ is an ideal.
%
%
%

If $R$ is acting on either $S$ or $\mathcal{D}$ in characteristic $0$, the proof can be simplified.  The use of \Cref{lem:prodrule} can be replaced by \Cref{lem:prodrulechar0} (for $S$) or \Cref{lem:prodruledivided} (for $\mathcal{D}$).
\end{proof}

\begin{rem}\label{rem:onlyHomogeneous}
Our interest is primarily in graded filtrations of homogeneous ideals, so we have stated \Cref{def:differentiallyclosedfiltration} and \Cref{thm:differentiallyclosed} for a \textit{filtration} of homogeneous ideals.  However, \Cref{def:differentiallyclosedfiltration} and \Cref{thm:differentiallyclosed} only use the hypothesis that $\I$ is a family of \textit{homogeneous} ideals.
\end{rem}

\subsection{Intersecting differentially closed graded filtrations}
In this section we describe how $\cL^s(\I)$ behaves under intersection of filtrations.  This will give us a number of additional examples of families of differentially closed graded filtrations.  Suppose $A$ is an index set and $\I_a=\{I_{a,n}\}_{n\in \N}$ is a filtration of ideals of $R$ for each $a\in A$.  We write $\cap_{a\in A} \I_a$ for the filtration $\{\cap_{a\in A} I_{a,n}\}_{n\in\N}$.

\begin{prop}\label{prop:IntersectionToSum}
	Suppose $A$ is an index set and $\I_a$ is a differentially closed graded filtrations of ideals for each $a\in A$.  Then
	\begin{enumerate}
		\item\label{ITS1} $\cap_{a\in A}\I_a$ is a differentially closed graded filtration of ideals and
		\item\label{ITS2} $\cL^s(\cap_{a\in A} \I_a)=\sum_{a\in A} \L^s(\I_a)$
	\end{enumerate}
\end{prop}
\begin{proof}
	
	For~\eqref{ITS1}, it is clear that $\cap_{a\in A}\I_a$ is a filtration.  We show that it is graded.  Given any two positive integers $m,n$, suppose $f\in \left(\cap_{a\in A} \I_a \right)_m=\cap_{a\in A} I_{a,m}$ and $g\in  \left(\cap_{a\in A} \I_a \right)_n=\cap_{a\in A} I_{a,n}$.  Since $\I_a$ is a graded family for every $a\in A$, $fg\in I_{a,m+n}$ for every $a\in A$ and thus $fg\in \cap_{a\in A} I_{a,m+n}=  \left(\ \cap_{a\in A} \I_a \right)_{m+n}$.  Now we show that $\cap_{a\in A}\I_a$ is differentially closed.  Suppose $f\in \left(\cap_{a\in A} \I_a \right)_{n}=\cap_{a\in A} I_{a,n}$ Since $\I_a$ is a differentially closed family for every $a\in A$, $D_{\ba}(f)\in I_{a,n-|\ba|}$ for every $a \in A$ and $D_{\ba}\in D^{n-1}_R$.  Therefore $D_{\ba}(f)\in\cap_{a\in A} I_{a,n}$.
	
	Now we prove~\eqref{ITS2}.  Since the construction of $\L^s$ is accomplished by putting together graded pieces, it suffices to show that
	\[
	\L^s(\cap_{a\in A} \I_a)_d=\sum_{a \in A} \L^s(\I_a)_d,
	\]
	for any integer $d\ge 0$.  From~\Cref{def:LTransform}, it suffices to show that
	\[
	\left(\cap_{a\in A} (I_{a,n})_d\right)^{\perp}=\sum_{a\in A} (I_{a,n}^\perp)_d.
	\]
	For any fixed $d\ge 0$, the intersection on the left hand side and the sum on the right hand side need only run over finitely many of the graded filtrations $\{\I_a\}_{a\in A}$ (since the intersection occurs in the finite dimensional vector space $R_d$ and the sum occurs in the finite dimensional vector space $S_d$).  Then the equality follows the fact that $(U\cap V)^\perp=U^\perp+V^\perp$ for any vector subspaces $U,V\subset R_d$~\cite[Lemma~2.7]{Ger95}.
\end{proof}

\begin{ex}\label{ex:subadditivepowers}
	Let $A$ be an index set, $I_a$ an ideal of $R$, and $\{r_{a,n}\}_{n\in\N}$ an increasing subadditive sequence for every $a\in A$.  For an ideal $I_a\subset R$, consider the filtration $\I_a=\{I_a^{r_{a,n}}\}_{n\in\N}$.  Since $\{r_{a,n}\}_{n\in\N}$ is increasing, this filtration is differentially closed by~\Cref{ex:PowersDiffClosed}.  It is graded because
	\[
	I_a^{r_{a,i}}I_a^{r_{a,j}}=I_a^{r_{a,i}+r_{a,j}}\subset I_a^{r_{a,i+j}},
	\]
	where the final containment follows because $r_{a,i+j}\le r_{a,i}+r_{a,j}$.  Thus $\{I_a^{r_{a,i}}\}_{n\in\N}$ is a differentially closed graded filtration for every $a\in A$.  It follows from Proposition~\ref{prop:IntersectionToSum} that $\I=\cap_{a\in A} \I_a$ is a differentially closed graded filtration, $\cL^s(\I)$ is an ideal, and $\cL^s(\I)=\sum_{a\in A} \cL^s(\I_a)$.  If $\{r_{a,n}\}_{n\in\N}$ is simply an increasing sequence for every $a\in A$, $\I_a$ is a differentially closed filtration, but not necessarily graded.  The same conclusions still follow, but we may lose a reciprocity for asymptotic growth factors that we explore in~\Cref{ss:DualSequencesDifferentiallyClosedFiltrations}.
\end{ex}

\begin{ex}\label{ex:ZariskiMainLemmaHolomorphic}
	If $\kk$ is a field and $I$ is a radical ideal of $R=\kk[x_0,\ldots,x_N]$, we show that $\I=\{I^{(n)}\}$ is a differentially closed graded filtration.  Let ${\rm Max}(R)$ be the collection of maximal ideals of $R$.  According to Zariski's Main Lemma on Holomorphic Functions~\cite{Zariski-1949} (see also~\cite[Theorem~2.12]{DDSGHN2018}), $I^{(n)}=\bigcap\limits_{\substack{\fm\in {\rm Max}(R)\\ I\subset \fm}} \fm^n.$  The conclusion now follows from~\Cref{ex:PowersDiffClosed} and~\Cref{prop:IntersectionToSum}.
\end{ex}

\begin{ex}\label{ex:LsForSymbolicPowers}
	In this example we state the main results of~\cite{EI95} in terms of the notation we have introduced.  Let $\kk$ be an algebraically closed field.  We can build on~\Cref{ex:ZariskiMainLemmaHolomorphic} to compute $\cL^s(\I)$ where $\I=\{I(X)^{(n)}\}\subset R$ consists of the symbolic powers of the ideal of a projective variety $X\subset\P^N$.  If $p$ is a point in $X$, write $\fm_p\subset R$ for the ideal of $p$.  In this context, Zariski's Main Lemma on Holomorphic Functions reads
	\[
	I(X)^{(n)}=\bigcap_{p\in X} \mathfrak{m}_p^n.
	\]
	Put $\cL^s(X)=\cL^s(\I)$.  For a point $p\in X$ write $\I_p=\{\mathfrak{m}_p^n\}_{n\in\N}$ and $\cL^s(p)=\cL^s(\I_p)$.  From~\Cref{ex:LsForPointIdeals}, $\cL^s(p)=\langle L_p^{s+1}\rangle\subset S$ if we consider the action of $R$ on $S$ and $\cL^s(p)=\langle L_p^{[c]}: c\ge s+1\rangle$ if we consider the action of $R$ on $\mathcal{D}$.

 \Cref{prop:IntersectionToSum} yields that $\cL^s(X)=\sum_{p\in X}\cL^s(p)$.  Thus we obtain
\[
\cL^s(X)=\langle L_p^{s+1}: p\in X\rangle
\] 
for the action of $R$ on $S$ in characteristic $0$.  In arbitrary characteristic, for the action of $R$ on $\mathcal{D}$, we have
\[
\cL^s(X)=\langle L_p^{[c]}: p\in X,c\ge s+1\rangle.
\]
\end{ex}

\begin{ex}
Suppose $\{I_n\}_{n\ge 1}$ is a differentially closed graded filtration in $R$ and $J\subset R$ is an ideal.  We leave it to the reader to verify that $\{I_n:J^\infty\}_{n\ge 1}$ is also a differentially closed graded filtration.  This gives yet another way to see that symbolic powers are differentially closed, since symbolic powers may be obtained by saturating ordinary powers with respect to an appropriate ideal $J$, and we have seen in~\Cref{ex:PowersDiffClosed} that ordinary powers form a differentially closed graded filtration.
\end{ex}

\subsection{Dual sequences for a differentially closed graded filtration}\label{ss:DualSequencesDifferentiallyClosedFiltrations}
We now return to duality of sequences.  One of the sequences we study is the sequence $\alpha(I_n)$ for a graded family of ideals $\{I_n\}_{n\ge 1}$.  The next lemma begins our study of the interaction of this sequence with $\cL^s(\I)$.

\begin{lem}\label{lem:FiniteLength}
Suppose $\I=\{I_n\}_{n\ge 1}$ is a differentially closed graded filtration of ideals, and put $\alpha_n=\alpha(I_n)$.  The following are equivalent:
\begin{itemize}
\item $S/\cL^s(\I)$ (respectively $\mathcal{D}/\cL^s(\I)$) has finite length,
\item $\alpha_n > n+s$ for all $n$ large enough.
\end{itemize}
In particular, $\cL^s(\I)$ has finite length for all $s\ge 1$ if and only if $\{\alpha_n-n\}_{n\in\N}$ is not bounded above.
\end{lem}
\begin{proof}
Fix a positive integer $s$ and put $\cL^s=\cL^s(\I)$.  Suppose $S/\cL^s$ has finite length.  Then
\[
\left(\frac{S}{\cL^s(\I)}\right)_{n+s}=0
\]
for all $n$ large enough.  By~\eqref{eq:SmodLvsI}, $(I_{n})_{n+s}=0$ and hence $\alpha(I_n)>n+s$ for all $n$ large enough.

Now suppose $\alpha(I_n)>n+s$.  Then $(I_n)_{n+s}=0$, hence $(S/\cL^s(\I))_n=0$.  If this holds for all $n$ large enough, $S/\cL^s(\I)$ clearly has finite length.  The proof for $\mathcal{D}$ is identical.
\end{proof}

\begin{rem}\label{ex:LsFiniteLengthProjectiveVariety}
When $X\subset \P^N$ is a projective variety in characteristic $0$, we claim that $S/\cL^s(X)$ has finite length if and only if $X$ is non-degenerate (meaning $X$ is not contained in a hyperplane).  To see this, note that if $X$ is contained in a hyperplane defined by $\ell=0$ for some linear form $\ell\in R$, then $\ell^n\in I(X)^{(n)}$ for all $n\in\N$.  Since we cannot have $\alpha(I(X)^{(n)})<n$, we have $\alpha(I(X)^{(n)})=n$ and $\alpha(I(X)^{(n)})-n= 0$ for all $n\in\N$.  Thus $S/\cL^s(X)$ does not have finite length by Lemma~\ref{lem:FiniteLength}.  On the other hand, suppose $X$ is non-degenerate.  Then $X$ contains points $p_0,\ldots,p_N$ which span $\P^N$.  By Example~\ref{ex:LsForSymbolicPowers}, $\cL^s(X)$ contains the ideal $\langle L_{p_0}^{s+1},\ldots,L_{p_N}^{s+1}\rangle$.  Since these are linearly independent, we may change coordinates so that $L_{p_0}=y_0,\ldots,L_{p_N}=y_N$.  Since $S/\langle y_0^{s+1},\ldots,y_N^{s+1}\rangle$ has finite length, so does $S/\cL^s(X)$. In arbitrary characteristic, we also have a similar result that $\mathcal{D}/\cL^s(X)$ has finite length if and only if $X$ is non-degenerate. The proof is the same as that in the case of characteristic $0$.  Notice that by Example~\ref{ex:LsForSymbolicPowers}, $\cL^s(X)$ contains the ideal $\langle L_{p_0}^{[c]},\ldots,L_{p_N}^{[c]}, c\ge s+1\rangle$, and it is clear that $\mathcal{D}/\langle y_0^{[c]},\ldots,y_N^{[c]}, c\ge s+1\rangle$ has finite length. 
\end{rem}

The second sequence we will study is the largest non-zero degree of $\mathcal{D}/\cL^s(\I)$, which we call the \textit{end} of $\mathcal{D}/\cL^s(\I)$. That is,
\[
\soc\left(\frac{\mathcal{D}}{\cL^s(\I)}\right)=\max\left\lbrace d : \left(\frac{\mathcal{D}}{\cL^s(\I)}\right)_d\neq 0\right\rbrace.
\]
Similarly, in characteristic $0$, we define $\soc(S/\cL^s(\I))$ as the largest non-zero degree of this quotient.  If $\soc(S/\cL^s(\I))<\infty$ then it is well known that $\soc(S/\cL^s(\I))=\reg(S/\cL^s(\I))$, where the latter is the Castelnuovo-Mumford regularity of $S/\cL^s(\I)$.

\begin{rem}
\label{rem:betasuperadditive}
Let $\I$ be a graded family of ideals. The sequence $\beta_s=\soc(S/\L^s(\I))$ (respectively $\beta_s=\soc(\mathcal{D}/\L^s(\I)))$ can be seen to be superadditive by interpreting it as
\[
\beta_s=\soc(S/\L^s(\I))=\max\{ d: (I_{d-s})_d\neq 0\} \text{ by } \eqref{eq:SmodLvsI}.
\]
The containment $(I_{d-s})_d(I_{d'-t})_{d'}\subseteq (I_{d+d'-(s+t)})_{d+d'}$ thus implies $\b_s+\b_t\leq \b_{s+t}$.
\end{rem}

Below we give a more refined version of this observation.

\begin{thm}\label{thm:InverseSystemDualityGeneral}
Suppose $\I=\{I_n\}_{n\in\N}$ is a differentially closed graded family of proper homogeneous ideals in $R$.  Put $\alpha_n = \alpha(I_n)$ and $\beta_s=\soc(S/\L^s(\I))$ (respectively, $\beta_s=\soc(\mathcal{D}/\L^s(\I))$).  Assume that the sequence $\{\alpha_n-n\}$ is nondecreasing and not bounded above.  Then
\begin{enumerate}
\item\label{c:1} $\{\a_n-n\}_{n\in \N}$ is a nondecreasing subadditive sequence.
\item\label{c:2} $\{\b_s-s\}_{s\in \N}$ and $\{\b_s\}_{s\in \N}$ are nondecreasing superadditive sequences.
\item\label{c:3} $\b_s-s=(\overrightarrow{\a_n-n})_s$
\item\label{c:4} $\a_n-n=(\overleftarrow{\b_s-s})_n$
\end{enumerate}
Write $\wa=\wa(\I)=\lim_{n\to\infty}\frac{\alpha_n}{n}$ and $\wb(\I)=\lim_{s\to\infty} \frac{\beta_s}{s}$.  Then
\[
\wb=\frac{\wa}{\wa-1} \mbox{  and  } \wa=\frac{\wb}{\wb-1}.
\]
\end{thm}
\begin{rem}\label{rem:char0}
Under the hypotheses of~\Cref{thm:InverseSystemDualityGeneral} in characteristic $0$, $\{\alpha_n-n\}_{n\ge 1}$ is always non-decreasing, which we can see as follows.  Since $\I$ is differentially closed, if $f\in I_n$ is a homogeneous polynomial of degree equal to $\alpha(I_n)$ ($n>1$), then $\frac{\partial f}{\partial x_i}\in I_{n-1}$ for $i=0,\ldots,N$.  Thus $\alpha(I_n)<\alpha(I_{n+1})$ unless all partials of $f$ vanish.  This happens if and only if $f$ is constant, which is impossible since $\I$ consists of proper ideals.  Thus $\alpha(I_{n-1})-(n-1)\le \alpha(I_{n})-n$, which shows $\{\alpha_n-n\}_{n\ge 1}$ is non-decreasing.
\end{rem}

\begin{proof}
For~\eqref{c:1}, the sequence $\alpha_n$ is subadditive by Lemma~\ref{lem:WaldSub}.  Since $n$ is linear, $\alpha_n-n$ is also subadditive.

We next prove~\eqref{c:3}.  Set $\gamma_n=\a_n-n$. By Definition~\ref{def:LTransform}, we have
\[
\begin{array}{rl}
\b_s & =\soc(S/\L^s)\\[5 pt]
&= \max\{d:\alpha(I^{(d-s)})\le d\}\\[5 pt]
&=\max\{d-s:\alpha(I^{(d-s)})-(d-s)\le s\}+s\\[5 pt]
&= \max\{t:\alpha(I^{(t)})-t\le s\}+s\\[5 pt]
&=\overrightarrow{\gamma}_s+s,
\end{array}
\]
which proves~\eqref{c:3}.  Part~\eqref{c:4} follows immediately from~\eqref{c:3} and Theorem~\ref{lem:DualityProperties} (1).

For~\eqref{c:2}, since $\beta_s-s=(\overrightarrow{\a_n-n})_s$ by~\eqref{c:3}, the definition of the transform $\overrightarrow{\a_n-n}$ implies $\{\b_s-s\}$ is also nondecreasing.  That $\{\beta_s-s\}$ is superadditive follows from~\eqref{c:3},~\eqref{c:1}, and \Cref{lem:DualityProperties} (3). Clearly, we have $\beta_s=(\beta_s-s)+s$. Since each of the sequences $\{\beta_s-s\}_{s\in\N}$ and $\id=\{s\}_{s\in \N}$ are nondecreasing and superadditive, the same is true of their sum, $\ub$.

Finally we prove the last two equalities, which are clearly equivalent.  By Theorem~\ref{lem:DualityProperties} (3), the desired result follows by means of the identity
\[
\wb-1=\lim_{s\to \infty} \frac{\b_s-s}{s}=\left(\lim_{n\to \infty} \frac{\a_n-n}{n}\right)^{-1}=\frac{1}{\wa-1},
\]
which is equivalent to the claims regarding $\wa$ and $\wb$.  The proof is identical for $\beta_s=\soc(\mathcal{D}/\cL^s(\I))$.
\end{proof}

\begin{cor}\label{cor:boundingWaldschmidt}
With the same setup as Theorem~\ref{thm:InverseSystemDualityGeneral}, we have inequalities
\[
\soc\left(\frac{\mathcal{D}}{\L^s(\I)}\right)\le\frac{s\wa}{\wa-1} \qquad  \text{ and }  \qquad 
\frac{n\wb}{\wb-1}\le \alpha(I_n).
\]
The same statements follow with $\mathcal{D}$ replaced by $S$.
\end{cor}
\begin{proof}
Put $\alpha_n=\alpha(I_n)$ and $\beta_s=\soc(\mathcal{D}/\L^s(\I))$.  By Theorem~\ref{thm:InverseSystemDualityGeneral}, $\wb=\frac{\wa}{\wa-1}$.  Since $\beta_s$ is superadditive, $\frac{\beta_s}{s}\le \wb=\frac{\wa}{\wa-1}$.  Multiplying both sides by $s$ gives the first inequality.  Likewise, by Theorem~\ref{thm:InverseSystemDualityGeneral}, $\wa=\frac{\wb}{\wb-1}$.  Since $\alpha_n$ is subadditive, $\frac{\alpha_n}{n}\ge \wa=\frac{\wb}{\wb-1}$.  Multiplying by $n$ gives the second inequality.  The proof is identical for $S$.
\end{proof}

\begin{ex}\label{ex:WaldschmidtDuality}
Continuing from Example~\ref{ex:LsForSymbolicPowers},
we consider the ideal $I(X)$ of a projective variety $X\subset \P^N$ in characteristic $0$, the filtration $\I=\{I(X)^{(n)}\}$, and the ideal $\cL^s(X)=\langle L_p^{s+1}: p\in X\rangle\subset S$.  We assume $X$ is non-degenerate so that $S/\cL^s(X)$ has finite length by Example~\ref{ex:LsFiniteLengthProjectiveVariety}. Following~\Cref{thm:InverseSystemDualityGeneral}, put $\alpha_n=\alpha(I(X)^{(n)})$, $\beta_s=\reg(S/\cL^s(X))$.  Then $\wa$ is the Waldschmidt constant of $I(X)$ (see~\Cref{def:waldschmidt}).  Theorem~\ref{thm:InverseSystemDualityGeneral} yields that the Waldschmidt constant can be expressed in terms of $\wb$ (and vice-versa).  Moreover, Corollary~\ref{cor:boundingWaldschmidt} yields the bounds
\[
\frac{n\wb}{\wb-1}\le \alpha(I(X)^{(n)}) \quad\mbox{and}\quad \reg\left(\frac{S}{\cL^s(X)}\right)\le \frac{s\wa}{\wa-1}.
\]
The right-hand bound was observed in~\cite{DiPVill21}, where it was used to determine a lower bound for the dimension of certain multivariate spline spaces.
\end{ex}

\section{Asymptotic regularity and Seshadri constant}

Throughout this section we consider  a finite set of points  $X=\{p_1,\ldots,p_r\}\subset \mathbb{P}^N$ and denote by $I(X)\subseteq R=\kk[x_0,\ldots,x_N]$ the saturated ideal defining $X$ with its reduced scheme structure.
\begin{defn}
\label{def:Seshadri}
The  multipoint Seshadri constant for $X$ is the real number
\[
\e(X)= \inf_C\left\{ \frac{\deg(C)}{ \sum_{i=1}^r {\rm mult}_{p_i} C} \right\} 
\]
where $C$ is any curve with $C\cap X\neq \emptyset$, $\deg(C)$ is the multiplicity of $R/I(C)$, and ${\rm mult}_{p_i} C$ is the multiplicity of $C$ at $p_i$, that is, the multiplicity of the local ring $(R/I(C))_{P_i}$, where $P_i=I(p_i)$. It suffices in fact to consider irreducible curves in the definition. Since we only consider Seshadri constants of varieties $X\subseteq \P^N$ with respect to the line bundle $ \mathcal{O}_{\P^n}(1)$, we suppress  this information from the notation.
\end{defn}

Seshadri constants were introduced in \cite{Demailly}. For nice expositions of the circle of ideas this has led to in the intervening years see \cite[\S 5.1]{Laz04} or \cite{primer}.

 In this section we establish a limit description for the multipoint Seshadri constant $\e(X)$. This generalizes a similar result in \cite[Theorem 5.1.17]{Laz04} for single point Seshadri constants. Moreover we establish a duality between the sequence of jet separation indices, whose limit is the multipoint Seshadri constant, and the sequence of Castelnuovo-Mumford regularities of the symbolic powers for the ideal $I(X)$; see \Cref{prop:superadd}. We further demonstrate how this duality underpins the well known reciprocity between the Seshadri constant $\e(X)$ and the asymptotic regularity (alternately termed the $s$-invariant) of $X$; see \cite[Remark 5.4.3]{Laz04}. Our methods recover this reciprocity relation; see Theorem \ref{thm:areg} for the specific statement 
 
A relevant sequence for our purposes requires the following definition.

\begin{defn}
\label{def:genjets}
Let $I$ be a homogeneous ideal of a standard graded ring $R$ with homogeneous maximal ideal $\fm$ and $d\in \N$. Define the {\em jet separation sequence} of $I$ by
\[
s(I,d)=\sup\{k \in \N\mid \reg(R/I^{(k+1)})\leq d\}.
\] 
\end{defn}

The terminology ``jet separation sequence" is justified by the following notion previously developed in the literature; see \cite[Definitions 5.1.15 and 5.1.16]{Laz04} building on the related notion of $k$-jet ampleness; see \cite{BS}.

\begin{defn}
\label{def:jets}
A finite set of points  $X=\{p_1,\ldots, p_r\}\subset \mathbb{P}^N$ with defining ideals $P_1, \ldots, P_r$ is said to {\em separate (uniform) $k$-jets in degree $d$} if the following  map  obtained by canonical projection onto each direct summand is surjective \begin{equation}
\label{eq:jetsep}
\kk[x_0,\ldots,x_n]_d \to \bigoplus_{i=1}^r \left(\kk[x_0,\ldots,x_n]/P_i^{k+1}\right)_d
\end{equation}
We define the {\em  jet separation index} of $X$ in degree $d$ to be the integer
\[
s(X,d)=\sup\{ k \in \N \mid X \text{ separates } k\text{-jets in degree } d\}.
\]
\end{defn}

The name coincidence gives an indication that the two notions defined above are related, a fact that we make precise in the next proposition.


\begin{prop}
\label{prop:jet}
Let $X$ be a finite set of $r\geq 2$ points in  $\mathbb{P}^N$ with defining ideal $I(X)$ and $N\geq 2$. Then for each $d\in\N$ the jet separation indices of \Cref{def:genjets} and \Cref{def:jets} agree, that is,  $s(X,d)= s(I(X),d)$.
\end{prop}
\begin{proof}
In geometric language the map \eqref{eq:jetsep} can be written as
\begin{equation}
\label{eq:H1}
H^0(\P^N, \mathcal{O}_{\P^n}(d) )\to H^0(\P^N,    \mathcal{O}_{\P^N}(d)/\fm_1^{k+1}) \oplus \cdots \oplus  H^0(\P^N,  \mathcal{O}_{\P^N}(d)/\fm_r^{k+1}),
\end{equation}
where $\fm_i$ is the ideal sheaf corresponding to $P_i$.
A necessary and sufficient condition for the surjectivity of \eqref{eq:H1} is
 $H^1(\P^N,  \I^{(k+1)}\otimes \mathcal{O}_{\P^N}(d))=0$, where $\I^{(k+1)}$ is the ideal sheaf corresponding to $I(X)^{(k+1)}$. 
This follows from the long exact sequence in homology arising from the short exact sequence of sheaves
\[
0\to  \I^{(k+1)}\otimes \mathcal{O}_{\P^N}(d) \to \mathcal{O}_{\P^N}(d) \to \mathcal{O}_{\P^N}(d)/\fm_1^{k+1}\otimes \cdots \otimes \mathcal{O}_{\P^N}(d)/\fm_k^{k+1} \to 0
\]
and the vanishing of $H^1(\P^N, \mathcal{O}_{\P^N}(d))$ due to $N\geq 2$.
Expressing regularity in terms of local cohomology (see \Cref{def:reg}) yields
\begin{eqnarray*}
\reg(R/I(X)^{(k+1)}) &=& \soc H^1_\fm(R/I(X)^{(k+1)})+1= \soc H^2_\fm(I(X)^{(k+1)})+1\\
           &=& \min\{d \mid H^1(\P^n,  \I(X)^{(k+1)}\otimes \mathcal{O}_{\P^N}(d))=0\}.
\end{eqnarray*}
It follows that $\reg(R/I(X)^{(k+1)})\leq d$ if and only if  \eqref{eq:jetsep} is surjective in degree  $d$. Thus the claim follows by comparing \Cref{def:genjets} and \Cref{def:jets}.
\end{proof}

We can now relate the jet separation sequence of an ideal with the sequence of regularities of its symbolic powers in the style of \cref{s:sequences}.

\begin{thm}
\label{prop:superadd}
Let $I$ be a homogeneous ideal of a graded ring $R$, set $s_d=s(I,d-1)$ for $d\in \N$, and set $r_k=\reg(I^{(k+1)})$. Then 
\begin{enumerate}
\item the sequences $\{s_d\}_{d\in\N}$ and $\{r_k\}_{k\in\N}$ are nondecreasing and dual as follows: 
\[  s_d=\overrightarrow{r}_d \text{ and } r_k=\so_{k}. \]
\item If $I$ is  aspCM, the sequence $\{r_k\}_{k\in\N}$ is subadditive and $\{s_d\}_{d\geq r_1}$ is superadditive. 
\item In particular the shifted jet separation sequence $s(X)[-1]=\{s(X,d-1)\}_{d\geq \reg{I(X)^{(2)}}}$ for a finite set of points $X$ in $\P^N$ with $N\geq 2$ is superadditive.
\item If $I$ is aspCM, the asymptotic regularity of $I$ is related to the asymptotic growth of the jet separation sequence by 
\[
\lim_{d\to\infty} \frac{s(I,d)}{d}=\widehat{\reg}(I)^{-1}
\] 
\end{enumerate}
\end{thm}
\begin{proof}

We have directly from \Cref{def:genjets} that
\[
s_d=\sup\{ k\mid \reg(R/I^{(k+1)})\leq d-1\}=\sup\{ k\mid \reg (I^{(k+1)})\leq d\}=\sup\{k \mid r_k\leq d\}.
\]
This establishes  the first part of claim (1) as well as giving that $\{s_d\}_{d\in\N}$ is nondecreasing. Applying the operator $\overleftarrow{}$ to the identity  $s_d=\overrightarrow{r}_d$ and using \Cref{lem:DualityProperties} (1) yields $\overleftarrow{s}_k=\overleftarrow{\overrightarrow{r}}_k=r_k$. It follows from the definition of $\overleftarrow{s}_k$ that $\{r_k\}_{k\in\N}$ is nondecreasing.

To establish the remaining claims, recall from \Cref{lem:CM} that for aspCM $I$ the sequence $\{\reg(I^{(k)})\}_{k\in\N}$ is subadditive and the sequence $\{r_k\}_{k\in\N}$ is the subsequence $\{\reg(I^{(k)})\}_{k\in\N}[1]$. Since $\{r_k\}_{k\in\N}$ is nondecreasing the same is true of $\{\reg(I^{(k)})\}_{k\in\N}$. Thus, \Cref{lem:subseq}~(1) allows to conclude that $\{r_k\}_{k\in\N}$ is subadditive.


Having established that $\{r_k\}_{k\in\N}$ is subadditive and that $s_d=\overrightarrow{r}_d$, we deduce that $\{s_d\}_{d\geq r_1}$ is a superadditive sequence of natural  numbers by \Cref{lem:DualityProperties}~(3). Claim (3) follows from (2) by means of \Cref{prop:jet}. 

Claim (4) follows from \Cref{lem:DualityProperties}~(4) and  part (2) of the current proposition, which yield $\widehat{s}=\widehat{r}^{-1}$. Combining this with identities $\widehat{\reg}(I)=\widehat{r}$ and $\lim_{d\to\infty}\frac{s(I,d)}{d}=\widehat{s}$ deduced from \Cref{lem:subseq}~(3), we obtain
\[
\lim_{d\to\infty}\frac{s(I,d)}{d}=\widehat{s}=\widehat{r}^{-1}=\widehat{\reg}(I)^{-1}.
\]
\end{proof}

Our next goal is to relate the multipoint Seshadri constant $\e(X)$ to the asymptotic growth of the jet separation sequence for $I(X)$. For this we will need a multipoint analogue of the well-known Seshadri criterion \cite[Theorem 1.4.13]{Laz04}, which we include for lack of a suitable reference.

\begin{prop}[Multipoint Seshadri criterion]
\label{prop:ample}
Consider  a finite  set of points  $X=\{p_1,\ldots, p_r\}\subset \mathbb{P}^N$ with $N\geq 2$ and let $B$ be the blowup of $\mathbb{P}^N$ at $X$ with projection map $\mu:B \to \mathbb{P}^N$ and exceptional divisor $E=\sum_{i=1}^r E_i$. Let $H=\mu^*(\mathcal{O}_{\P^N}(1))$. Then the Seshadri constant  of \Cref{def:Seshadri} can be alternatively described as
\begin{eqnarray}
\label{eq:Seshadriample}
\e(X) &=& \sup \left\{ \frac{p}{q} : p,q\in \Q_{>0},  \ qH-p E \text{ is ample}\right \} \nonumber\\
&=&\sup \left\{ \lambda\in\R :  H-\lambda E \text{ is ample}\right \}
\end{eqnarray}
\end{prop}
\begin{proof}
Temporarily denote $\e'(X):=\sup \left\{ \frac{p}{q} : p,q\in \Q_{>0},  \ qH-p E \text{ is ample}\right \}$.
For $p,q\in\Q_{>0}$ set $L_{p,q}:=qH-p E$ to be a $\Q$-divisor on $B$. Suppose $L_{p,q}$ is ample and hence nef. Computing the intersection product with the pullback of a curve $C\subset \P^N$ gives
\[
L_{p,q}\cdot \mu^*(C)=(qH-p E)\cdot \mu^*(C)= q\deg(C)-p\left( \sum_{i=1}^r {\rm mult}_{p_i}(C)\right)\geq 0
\]
and hence $\e(X)\geq \frac{p}{q}$ by \Cref{def:Seshadri}. We conclude that $\e(X)\geq \e'(X)$.

Conversely, suppose $\frac{p}{q}\leq \e(X)$. We show  that $L_{p,q}$ is nef. If $D$ is a curve in $B$, then $D=D_1+D_2$ with $D_1$ contained in $E$ and $D_2$ not contained in $E$ (we allow the possibility that $D_1=0$ or $D_2=0$). We have that $E\cdot D_1=- \deg(N_{E/B}|_{D_1})\leq0$ (where $N_{E/B}$ is the normal bundle of the exceptional divisor) and $\mu(D_2)=C$ is a curve in $\P^N$, thus by a computation similar to the  above display we conclude from $\frac{p}{q}\leq \e(X)$ that $L_{p,q}\cdot D_2\geq 0$. Therefore we have
\[
L_{p,q} \cdot D=(qH-p E)\cdot D= -p E\cdot D_1+L_{p,q}\cdot D_2 \geq 0,
\]
which shows $L_{p,q}$ is nef.
Suppose now that $\frac{p}{q}<\e(X)$. By  \cite[Section II, Proposition 7.10]{Hartshorne} the divisor $L_{1,d}=dH-E$ is ample for $d\in \N, d\gg0$. Fix such a $d$. Since $\frac{p}{q}<\e(X)$ and the expression $\frac{p-\delta}{q-\delta d}$ is a continuous function of $\delta\in\R_{>0}$ one can find $\delta\in \Q, \delta>0$ so that $\frac{p-\delta}{q-\delta d}\leq \e(X)$. Then the identity
\[
L_{p,q}= L_{p-\delta ,q-\delta d}+\delta L_{1,d}
\]
shows that $L_{p,q}$ is ample since  $L_{p-\delta ,q-\delta d}$ is nef by the above considerations, $\delta>0$,  and $L_{1,d}$ is ample; see \cite[Corollary 1.4.10]{Laz04}. We have obtained that $\e'(X)\leq \e(X)$, hence the first equality in \eqref{eq:Seshadriample} is established.
The second equality follows from the first noting that $L_{p.q}$ is ample if and only if $L_{1,p/q}=H-\frac{p}{q}E$ is ample and hence the last set in the display \eqref{eq:Seshadriample} is the closure of the first in the topology on $\R$. 
\end{proof}


The following is a multipoint version of \cite[Theorem 5.1.17]{Laz04}. Our proof follows the single point case closely, however the prior knowledge that the limit in the statement exists as a consequence of \Cref{prop:superadd} allows for slight simplifications. 

\begin{thm}
\label{prop:e}
If $X$ is a finite set of $r$ points in $\P^N$, $N\geq 2$, the limit of the jet separation index sequence exists and is equal to the multi-point Seshadri constant
\[
\e(X)= \lim_{d \to \infty} \frac{s(X,d)}{d}.
\]
\end{thm}
\begin{proof}

Set $d_0=\reg(I(X)^{(2)})$. By \Cref{prop:superadd}~(3) we have that $\underline{s}=\{s(X,d-1)\}_{d\geq d_0}$ is a superadditive sequence of natural numbers. Using  \Cref{lem:subseq}~(3) applied to $s[1]=\{s(X,d)\}_{d\geq d_0-1}$, we see that $\lim_{d\to \infty} \frac{s(d,X)}{d}$ exists. 

Suppose $X=\{p_1,\ldots, p_r\}$ with $I(p_i)=P_i$. Let $C$ be an irreducible curve that contains at least one point  $p_{i_0}\in X$. Assume $d\geq d_0-1$ and set $k= s(X,d)$. Take $\overline{F_i}\in P_i^k/P_i^{k+1}$, one for each $p_i\in X$, so that the image of $\overline{F_{i_0}}$ in $\frac{{P_{i_0}}^k}{P_{i_0}^{k+1}}\otimes_R \frac{R}{I(C)}$ is nonzero. This is possible since\[
\frac{{P_{i_0}}^k}{P_{i_0}^{k+1}}\otimes_R \frac{R}{I(C)}\cong\frac{\overline{P_{i_0}}^k}{\overline{P_{i_0}}^{k+1}}\neq 0 , \text{ where } \overline{P_{i_0}}=\frac{P_{i_0}}{I(C)},
\]
in view of  $\overline{P_{i_0}}^k\neq \overline{P_{i_0}}^{k+1}$ by Krull's intersection theorem.
Due to the surjectivity of the map in equation \eqref{eq:jetsep} recalled below
\[
\kk[x_0,\ldots,x_n]_d \to \bigoplus_{i=1}^r \left(\kk[x_0,\ldots,x_n]/P_i^{k+1}\right)_d ,
\]
 there exists $F\in R_d$ that maps to the tuple $(\overline{F_1}, \ldots, \overline{F_r})$. Since $\overline{F_i}\in P_i^k/P_i^{k+1}$, we have $F\in \bigcap_{i=1}^r P_i^k= I(X)^{(k)}$. Since $\overline{F_{i_0}}$ is nonzero modulo $I(C)$ we have $F\in I(X)^{(k)}\setminus I(C)$. Since $I(C)$ is prime,  $F$ is regular on $R/I(C)$ and thus the associativity formula for multiplicities provides an inequality
\[
\deg(C)\cdot d=e\left(\frac{R}{I(C)+F}\right)\geq \sum_{i=1}^r {\rm mult}_{p_i} C \cdot {\rm mult}_{p_i} F \geq k \left(\sum_{i=1}^r {\rm mult}_{p_i} C \right).
\]
We have shown for each curve $C$ with $C\cap X\neq \emptyset$ and each $d\geq d_0-1$ that the following inequality holds
\[
\frac{\deg(C)}{\sum_{i=1}^r {\rm mult}_{p_i} C} \geq\frac{s(X,d)}{d}, \text{ thus } \e(X) \geq\frac{s(X,d)}{d}.
\]
Taking limits  we deduce
\begin{equation}
\label{eq:Seshadrigreater}
 \e(X) \geq\lim_{d\to \infty} \frac{s(d,X)}{d}.
\end{equation}

It remains to establish the opposite inequality to \eqref{eq:Seshadrigreater}. For this, fix integers $p, q$ with $0<\frac{p}{q}<\e(X)$ and let $\mu:B\to \P^N$ be the blow up of $\P^N$ at $X$, with exceptional divisor $E$ and $H=\mu^*(\mathcal{O}(1))$ as in \Cref{prop:ample}. Then $L_{p,q}=qH- pE$ is ample by the aforementioned result. By asymptotic Serre vanishing \cite[Chapter III, Proposition 5.3]{Hartshorne} we have that there exists $m_0\in\N$ so that 
\[
H^1(B, \mathcal{O}_B(mL_{p,q})) = 0 \text{ for } m\geq m_0.
\]
The leftmost homology group is in turn isomorphic to the one listed below,  by \cite[Lemma 4.3.16]{Laz04}, , where $\I^{(k+1)}$ is the ideal sheaf corresponding to $I(X)^{(k+1)}$. 
Its va\-ni\-shing
\[
H^1(\P^n,  \I^{(mp)}\otimes \mathcal{O}_{\P^N}(mq))=0
\]
indicates via the definition of regularity in terms of local cohomology with respect to the maximal homogeneous ideal $\fm$ or $R=k[\P^N]$ (see \Cref{def:reg}) that 
\begin{eqnarray*}
\reg(I^{(mp)})-1 &=& \reg(R/I^{(mp)})=\soc H^1_\fm(R/I^{(mp)})+1= \soc H^2_\fm(I^{(mp)})+1\\
           &=& \min\{d \mid H^1(\P^n,  \I^{(mp)}\otimes \mathcal{O}_{\P^N}(d))=0\} \leq mq,\\
\text{ so} & & \quad \frac{\reg(I^{(mp)})}{mp}<\frac{mq}{mp}=\frac{q}{p} \text{ for } m\gg0.
\end{eqnarray*}
Taking the limit as $m\to\infty$ we  obtain
$
\widehat{\reg}(I)\leq \frac{q}{p}
$. Equivalently, by \Cref{prop:superadd} it follows that 
\[
 \lim_{d \to \infty} \frac{s(X,d)}{d} =\widehat{\reg}(I)^{-1}\geq \frac{p}{q}.
\]
Replacing $\frac{p}{q}$ by a sequence of rational numbers that converges to $\e(X)$ shows that $ \lim_{d \to \infty} \frac{s(X,d)}{d}\geq \e(X)$ and completes the proof.
\end{proof}

The following corollary recovers a particular instance of the well-known reciprocity between the Seshadri constant and the asymptotic regularity noted in \cite[Remark 1.3 and Theorem B]{CEL01}. Our main contribution here is to show that this reciprocity holds for a very precise structural reason, that is, the duality of the sequences in \Cref{prop:superadd}. 

\begin{cor}
\label{thm:areg}
The asymptotic regularity of a finite set $X$ of $r\geq 2$ points in $\P^N$ with $N\geq 2$ is the reciprocal of the Seshadri constant. In symbols, we have
\[
\widehat{\reg}(I(X))=\e(X)^{-1}.
\]
\end{cor}
\begin{proof}
\Cref{prop:superadd} (5) together with \Cref{prop:jet} and \Cref{prop:e} yield the desired conclusion
\[
\widehat{\reg}(I(X))= \left(\lim_{d\to\infty} \frac{s(I(X),d)}{d}\right)^{-1}= \left(\lim_{d\to\infty} \frac{s(X,d)}{d}\right)^{-1}=\e(X)^{-1}.
\]
\end{proof}

\section{Homological reformulations of the Nagata--Iarrobino conjecture}
\label{s:Nagata}
In the following we refer to a {\em very general set of $r$ points} in $\P^N$ to mean outside countably many proper subvarieties of the symmetric product $\Sym^r(\P^N)$ of $\P^N$.

In \cite{Nagata59} Nagata  established the upper bound $\widehat{\alpha}(I(X))\leq \sqrt{r}$ for any set $X$ of $r\geq 9$ very general points in $\P^2$. He also proposed, in different language, the following conjecture to the effect that very general sets of points attain  the maximum value of the Waldschmidt constant permitted by this inequality.  

\begin{conj}[Nagata]
\label{conj:Nag}
Any set $X$ of $r\geq 10 $ very general points in $\P^2$ over a field of characteristic zero satisfies  
$\alpha(I(X)^{(n)})> n\sqrt{r}$  for all $n\in\N$.
Equivalently, there is an equality 
\[
\wa(I(X))=\sqrt{r}.
\]
\end{conj}
 
This statement holds true for $r$ a perfect square, by Nagata's work in  \cite{Nagata59}, but it remains open for
all other values of $r$. We comment on the equivalence of the two claims in the above conjecture. For $\sqrt{r} \not \in\N$ (the case that is still open), the conjectured inequality for initial degrees in  \Cref{conj:Nag} is equivalent to $\alpha(I(X)^{(n)})\geq n\sqrt{r}$.
Utilizing the known upper bound $\widehat{\alpha}(I_X)\leq \sqrt{r}$ and the description of the Waldschmidt constant as an infimum (see Definition \ref{def:waldschmidt}), we see that the two statements in Conjecture \ref{conj:Nag} are indeed equivalent.

Notable advances on  \Cref{conj:Nag} have been made in \cite{Xu, Tutaj,Harbourne01, HarbourneRoe08}, however in its full generality it currently seems out of reach. See \cite{CHMR13} for further information and possible generalizations of this long-standing conjecture. 
\Cref{conj:Nag} can be equivalently reformulated in terms of the Seshadri constant as
\begin{equation}
\label{eq:nagseshadri}
\e(X)=\frac{1}{\sqrt{r}}.
\end{equation}
The inequality $\e(X)\leq 1/\sqrt{r}$ is known to hold in $\P^2$; this is equivalent to the known upper bound $\widehat{\alpha}(I_X)\leq \sqrt{r}$  by the arguments in the proof of \Cref{prop:equivalent}.
Below we use this equivalence  to give further equivalent homological formulation of Nagata's conjecture. 
Intuitively, in homological terms this conjecture becomes the statement that the width of the Betti table of the symbolic powers of $I(X)$ grows sub-linearly.
 
 \begin{conj}
 \label{conj:homolNag}
 Any set $X$ of $r\geq 10$ very general points in $\P^2$ over a field of characteristic zero satisfies  
 \[
 \widehat{\alpha}(I(X))=\widehat{\reg}(I(X)),
 \quad 
\text{ equivalently }
\quad
 \lim_{n\to\infty} \frac{\reg(I(X)^{(n)})-\alpha(I(X)^{(n)})}{n}=0.
 \]
 \end{conj}

\noindent Iarrobino \cite{iarrobino97} generalized \Cref{conj:Nag} to projective spaces of arbitrary dimension.

\begin{conj}[Iarrobino]
\label{conj:Iar}
A set $X$ of $r$ very general points in the projective space $\P^N$ over a field of characteristic zero with $r\geq \max\{N+5, 2^N\}$ and $(r,N)\not \in\{(7,2), (8,2), (9,3)\}$ satisfies $\alpha(I(X)^{(n)})\geq n\sqrt[N]{r}$  for all  $n\in \N$. Equivalently, apart from the given list of exceptions, there is an equality 
\[
\wa(I(X))=\sqrt[N]{r}.
\]
\end{conj}

\Cref{conj:Iar} is known to hold only for the case $r=s^N$; see \cite{Evain}. 

Below we use our results to reformulate the above conjecture \Cref{conj:Iar} in homological terms  using inverse systems.

 \begin{conj}
 	\label{conj:inverseNag}
 Under the hypotheses of \Cref{conj:Iar} the following holds
	 	\[
 	\lim_{s\to\infty} \frac{\reg(S/\L^s(X))}{s}=\frac{\sqrt[N]{r}}{\sqrt[N]{r}-1},
 	\]
 	where $\L^s(X)=\langle L_{p_1}^{s+1},\ldots,L_{p_r}^{s+1}\rangle \subset S=\kk[y_0,\ldots, y_N]$.
 \end{conj}

 \begin{prop}
 \label{prop:equivalent}
 Conjectures \ref{conj:Iar} and \ref{conj:inverseNag} are equivalent. Moreover, conjectures \ref{conj:Nag} and \ref{conj:homolNag} are equivalent. 
 \end{prop}
 \begin{proof}
 The equivalence of Conjecture~\ref{conj:inverseNag} to Conjecture~\ref{conj:Iar} follows immediately from the duality of asymptotic invariants given by \Cref{thm:InverseSystemDualityGeneral}.

Now we show the equivalence of Conjectures~\ref{conj:Nag} and~\ref{conj:homolNag}.  From \Cref{def:Seshadri} one sees that there is an inequality relating the  Waldschmidt constant, and the  Seshadri constant
 \begin{equation}
 \label{eq:WaldSesh}
 \widehat{\alpha}(X) \geq r \e(X)
 \end{equation}
  Indeed, let $C$ be a curve in $\P^2$ with  $\deg(C)=\alpha(I^{(n)})$ and ${\rm mult}_{p_i} C=n$ for each $i$. Then one has $\e(X)\leq \frac{\alpha(I^{(n)})}{nr}$ by the definition of $\e(X)$ and the inequality follows by passing to the limit.
 While equality need not hold in \eqref{eq:WaldSesh} in general, remarkably equality does hold for  a very general set of points $X$; see \cite[Lemma 2.3.1]{BH10}; thus under the hypotheses of our conjectures we have $ \widehat{\alpha}(I(X))=r\e(X)$. 
 This justifies the equivalence of \Cref{conj:Nag} asserting  $\widehat{\alpha}(I(X))=\sqrt{r}$ and the identity \eqref{eq:nagseshadri} mentioned above.

 Rewriting the identity $ \widehat{\alpha}(I(X))=r\e(X)$ using  \Cref{thm:areg} yields 
 \[ \widehat{\alpha}(I(X))\cdot \widehat{\reg}(I(X))=r.\]
It follows that the claim $\widehat{\alpha}(I(X))=\sqrt{r}$ of  \Cref{conj:Nag} is equivalent to $\widehat{\reg}(I(X)) =\sqrt{r}$ and also equivalent to $\widehat{\alpha}(I(X))=\widehat{\reg}(I(X))$.  The second claim of \Cref{conj:homolNag} follows from feeding the definitions of these asymptotic invariants into the equality.
\end{proof}

\begin{ex}
Here we illustrate some of the exceptions to  \Cref{conj:Iar} and \Cref{conj:inverseNag}.
By contrast to  \Cref{conj:Iar}, which predicts irational values for the Waldschmidt constant whenever $\sqrt[N]{r}\not \in\N$ , the Waldschmidt constant and the asymptotic regularity for sets $X$ of  few general points in $\P^N$  are given by rational functions in the number of points,  in particular they are rational numbers given by
\[
\wa(I_X)=
\begin{cases}
\frac{r}{r-1} & \text{ if } \#X=r\leq N+1,\\
\frac{r}{r-2} & \text{ if } \#X=r=N+2,\\
\frac{r-1}{r-3} & \text{ if } \#X=r=N+3 \text { is even}, \\
\frac{r(r-2)}{r^2-4r+2} & \text{ if } \#X=r=N+3 \text{ is odd}.
\end{cases}
\]
See \cite[Proposition B.1.1]{DHST14} for the last three cases and \cite[Proposition 5.1]{NT19} for a more general result in this direction. Utilizing the formulas in \Cref{thm:InverseSystemDualityGeneral} we obtain the asymptotic growth factor for the regularity of the inverse systems $\L^s(Z)$  
\[
\lim_{s\to\infty} \frac{\reg(S/\L^s(X))}{s}=
\begin{cases}
r & \text{ if } \#X\leq N+1,\\
r/2 & \text{ if } \#X=r=N+2,\\
(r-1)/2 & \text{ if } \#X=r=N+3 \text { is even}, \\
r(r-2)/2(r-1) & \text{ if } \#X=r=N+3 \text{ is odd}.
\end{cases}
\]
The same result can be derived from \cite[Theorem 4.4 and Theorem 4.7]{NT19}. 
\end{ex}

\section{Closing comments and invitations for future work}
\label{s:open}

We close with a number of questions which arose in the process of our writing.  The first two questions concern the subadditivity of sequences associated to the symbolic powers of an ideal.  We saw in \cref{sec:SequencesFromGradedFamilies} that if $\nu:R\to\Z$ is an $R$-valuation then the sequence $\nu(I^{(n)})$ is subadditive for any ideal $I\subset R$, and we relate sequences of this form to the \textit{resurgence} $\rho(I)$ and \textit{asymptotic resurgence} $\wrho(I)$.  In \cref{sec:resurgence} we define
\[
\lambda_n(I)=\max\{d:I^{(d)}\not\subseteq I^n\}\quad\mbox{and}\quad\beta_n(I)=\max\{d:I^{(d)}\not\subseteq \overline{I^n}\},
\]
where $\overline{I^n}$ is the integral closure of $I^n$.  We have examples of ideals $I$ where $\lambda_n(I)$ is not superadditive since $\rho(I)=\lim_{n\to\infty} \frac{\lambda_n}{n}\neq \sup\{\lambda_n/n\}=\wrho(I)$.  However, the sequence $\beta_n$ necessarily satisfies $\lim_{n\to\infty}\frac{\beta_n}{n}=\sup\{\beta_n/n\}=\wrho(I)$ by \Cref{prop:betahat=rhohat}, which is one of the properties of a superadditive sequence.  Thus it seems natural to ask if $\beta_n(I)$ is a superadditive sequence.

\begin{question}\label{ques:superadditiveintegralclosure}
For an ideal $I$ in a regular ring $R$, is the sequence $\beta_n(I) =\max\{d:I^{(d)}\not\subseteq \overline{I^n}\}$ a superadditive sequence?
\end{question}

If $I$ is an ideal so that~\Cref{ques:superadditiveintegralclosure} has a negative answer, then the failure of containment $I^{(d)}\not\subseteq \overline{I^n}$ is necessarily detected by \textit{different} valuations as $n$ increases, which is an interesting behavior.

Our next question concerns the (Castelnuovo-Mumford) regularity of symbolic powers.  If all symbolic powers of an ideal $I$ are Cohen-Macaulay, \Cref{lem:CM} shows that $\reg(I^{(n)})$ is a subadditive sequence, while \Cref{ex:nonsubreg} shows that this sequence may not be subadditive even if $I$ is a squarefree monomial ideal.  This example is not so far from being subadditive, however, which leads us to the following question.

\begin{question}
For a radical ideal $I$ in a polynomial ring, is the sequence $\reg(I^{(n)})+K$ a subadditive sequence for some appropriate integer $K$?  In particular, is this true if $K$ is the number of variables in the polynomial ring?
\end{question}

In \Cref{ex:nonsubreg}, a calculation shows that $\reg(J(m,s)^{(t)})+K$ is subadditive for any $K\ge (m-2)(s-1)$; the number of variables in the ambient polynomial ring is $m(s+1)$.

Our next question concerns the differentially closed graded filtrations of ideals introduced in~\cref{sec:InverseSystems}.  If $\I=\{I_n\}_{n\ge 1}$ is a differentially closed graded filtration of ideals in $R$, we found in \Cref{thm:InverseSystemDualityGeneral} a duality between the sequences $\a_n=\alpha(I_n)$ and $\b_r=\soc(\mathcal{D}/\cL^r(\I))$ (with the contraction operation) or $\b_r=\reg(S/\cL^r)$ (with the differentiation action).  This duality of sequences arose from Macaulay-Matlis duality.  Following the discussion of \cref{sec:SequencesFromGradedFamilies}, we note that $\alpha(I_n)$ is a special case of the sequence $\nu(I_n)$ for an $R$-valuation $\nu:R\to \Z$.  In general, $\nu(I_n)$ is subadditive and, in case $I_n=I^{(n)}$ for a fixed ideal $I$, its asymptotic growth factor can be used to bound or find the asymptotic resurgence of $I$ (\Cref{lem:valuationLimit} and \Cref{prop:betahat=rhohat}).  With this in mind, we ask the following open-ended question.

\begin{question}
Suppose $\I=\{I_n\}_{n\ge 1}$ is a differentially closed graded family.  Does Macaulay-Matlis duality give a meaningful algebraic interpretation for the sequence $\overrightarrow{\nu(I_n)}$ for an arbitrary valuation $\nu:R\to\Z$, extending \Cref{thm:InverseSystemDualityGeneral}?  If not, can the valuation $\nu$ be used to \textit{twist} Macaulay-Matlis duality in a way that does give a meaningful interpretation of $\overrightarrow{\nu(I^{(n)})}$?  As in \Cref{thm:InverseSystemDualityGeneral}, we likely need to shift the sequence $\nu(I_n)$ appropriately.
\end{question}

If $\I=\{I_n\}$ consists of the symbolic powers of a radical ideal over an algebraically closed field, then Emsalem and Iarrobino~\cite{EI95} give a concrete description for the ideal $\cL^s(\I)$.  Inspired by their description, we pose the following question.

\begin{question}\label{ques:DualGens}
If $\I=\{I_n\}_{n\ge 1}$ is a differentially closed graded filtration, under what conditions can we give a concrete description of the generators of $\cL^s(\I)$?  Under what conditions do the generators have a geometric interpretation?
\end{question}

From the end of \cref{sec:InverseSystems}, we have a large pool of differentially closed graded filtrations for which we can ask \Cref{ques:DualGens}.

If $\I=\{I_n\}_{n\ge 1}$ is a graded family of monomial ideals, then one may associate to $\I$ its \textit{Newton-Okounkov body}~\cite{Ha-Nguyen-Okounkov}.  For instance, if $\I$ consists of the symbolic powers of a monomial ideal $I$, the Newton-Okounkov body of $\cI$ is the \textit{symbolic polyhedron} introduced in~\cite{CEHH}.  It is natural to ask if there is an appropriate \textit{dual body} for the family $\cL^s(\I)$.  We plan to address aspects of the following question in an upcoming paper.

\begin{question}
If $\I$ is a differentially closed graded family of monomial ideals, is there an associated convex body which encodes the monomials not in $\cL^s(\I)$?  If so, when do these convex bodies limit to a polyhedron (like the symbolic polyhedron)?  In what situations can we determine the bounding inequalities?
\end{question}

In \cref{s:Nagata}, we saw a number of reformulations of the Nagata conjecture concerning the Waldschmidt constant of at least $10$ very general points in $\P^2$.  \Cref{conj:homolNag} rephrases this conjecture as an equality of the Waldschmidt constant with the asymptotic regularity.  We ask which varieties $X$ satisfy sub-linear growth for the width of the betti table of $I(X)^{(n)}$.

\begin{question}\label{ques:ExtendingAsymptoticRegularity}
What varieties $X$ can \Cref{conj:homolNag} be extended to?  That is, for what varieties $X$ do we have the equality $\widehat{\alpha}(I(X))=\widehat{\reg}(I(X))$?
\end{question}

Let $\omega(I^{(n)})$ be the largest degree of a generator of $I^{(n)}$.  We always have $\alpha(I^{(n)}) \le \omega(I^{(n)})\le \reg(I^{(n)})$.  If $I=I(X)$ is the ideal of a variety answering \Cref{ques:ExtendingAsymptoticRegularity} positively, then $\omega(I^{(n)})-\alpha(I^{(n)})$ must also grow sub-linearly.  There are many ideals for which it is known that $\reg(I^{(n)})$ differs from $\omega(I^{(n)})$ by a constant independent of $n$ - for instance \textit{star configurations of hypersurfaces}~\cite{Mantero-2020}.  However, in the case of star configurations of hypersurfaces, $\omega(I^{(n)})-\alpha(I^{(n)})$ does \textit{not} grow in a sublinear fashion, hence star configurations of hypersurfaces do not satisfy $\widehat{\alpha}(I)=\widehat{\reg}(I)$.

If $I$ is a monomial ideal, then~\cite{CEHH} shows that $\widehat{\alpha}(I)$ is the minimum sum of the coordinates of a vertex of the symbolic polyhedron of $I$, while~\cite[Theorem~1.3]{DHNT21} shows that $\widehat{\reg}(I)$ is the maximum sum of the coordinates of a vertex of the symbolic polyhedron of $I$.  Thus \Cref{ques:ExtendingAsymptoticRegularity} has a positive answer for a $I$ a monomial ideal precisely when all vertices of the symbolic polyhedron have the same coordinate sum.  More concretely, \Cref{ques:ExtendingAsymptoticRegularity} has a positive answer for any monomial ideal $I=\langle x^{\alpha_1},\ldots,x^{\alpha_n}\rangle$ whose symbolic polyhedron has a unique maximal bounded face (under inclusion) which can be described as both:
\begin{itemize}
\item The convex hull of the vertices of the symbolic polyhedron or
\item The intersection of the symbolic polyhedron with a hyperplane of the form $|\alpha|=c$ for some rational number $c$.
\end{itemize}

For instance, both bullet points are satisfied if $I$ is the edge ideal of a bipartite graph (in this case it is known that the ordinary and symbolic powers coincide~\cite{GVV05}).  More generally both bullet points are satisfied if $I$ is a monomial ideal generated in a single degree and $I^{(n)}=\overline{I^n}$ for all $n\ge 1$ (for squarefree monomial ideals this is also related to the \textit{packing problem}~\cite{DDSGHN2018}).  We are not aware of an algebraic characterization for those monomial ideals which have a symbolic polyhedron whose vertices all have the same coordinate sum.

\begin{rem}
If $I$ is a squarefree monomial ideal which satisfies the two bullet points above, then we can show that $I$ is generated in a single degree and that the number $c$ in the second bullet point above is precisely the generating degree of $I$.  To prove this, we need only show that for a squarefree ideal $I$ there is at least one generator of $I$ whose exponent vector is a vertex of the symbolic polyhedron $\SP(I)$.

Recall that if $I$ is squarefree then there are monomial prime ideals $P_0,\ldots,P_k\subset R=\kk[x_0,\ldots,x_N]$ such that $P_i\not\subset P_j$ for any $1\le i,j\le k$, and $I=P_0\cap\cdots\cap P_k$.  Take a generator of $I$ which has minimal support; re-indexing the variables if necessary we may suppose that $M=x_0\ldots x_t$ is the product of the first $t+1$ variables of $R$.  Since $M$ has minimal support among generators of $I$, the monomial $M_i=M/x_i$ is not in $I$ for any $i=0,\ldots,t$.  Re-ordering the primes $P_0,\ldots,P_k$ if necessary, we may assume that $M_i\notin P_i$ for $i=0,\ldots,t$.  This implies that $P_i$ is generated by $x_i$ and some subset of the variables $\{x_{t+1},\ldots, x_N\}$ for $i=0,\ldots,t$.  Recall that the defining inequalities of $\SP(I)$ are given by $\sum_{x_i\in P_j} a_i \ge 1$ for $j=1,\ldots,k$ and $a_i\ge 0$ for $i=0,\ldots,N$.  Consider the system of equations given by $a_{t+1}=\ldots =a_N=0$ and $\displaystyle \sum_{x_j\in P_i} a_j = 1$, $i=0,\ldots ,t$.  Since $x_i \in P_{i}$ and $P_i$ is generated by a subset of $\{ x_i, x_{t+1}, \ldots ,x_N\}$, this system has a unique solution $a_i = 1$, $i=0,\ldots ,t$ and $a_{t+1}=\ldots =a_N=0$.  This is a vertex of $\SP(I)$ and is clearly the exponent vector of the monomial $M$, completing the proof.
\end{rem}

In this paper we have explored the sequence duality of \Cref{def:overunder} in the context of initial degree and regularity of symbolic powers (and more generally, differentially closed filtrations).  We close with the following invitation to the reader.

\begin{question}
In what other algebraic-geometric contexts do subadditive and superadditive sequences naturally appear?  For each such sequence, is there a meaningful algebraic interpretation for the dual sequence of \Cref{def:overunder}?
\end{question}

\paragraph{\bf Acknowledgements} We thank Anthony Iarrobino for his comments on this work. The third author is partially supported by NSF DMS--2101225.

\bibliographystyle{alpha}
\bibliography{bibl}

\newcommand{\etalchar}[1]{$^{#1}$}
\begin{thebibliography}{BDRH{\etalchar{+}}09}

\bibitem[Ati07]{Atiyah}
Michael Atiyah.
\newblock Duality in mathematics and physics.
\newblock Lecture notes from the Institut de Matematica de la Universitat de
  Barcelona, 2007.

\bibitem[BDRH{\etalchar{+}}09]{primer}
Thomas Bauer, Sandra Di~Rocco, Brian Harbourne, Micha\l Kapustka, Andreas
  Knutsen, Wioletta Syzdek, and Tomasz Szemberg.
\newblock A primer on {S}eshadri constants.
\newblock In {\em Interactions of classical and numerical algebraic geometry},
  volume 496 of {\em Contemp. Math.}, pages 33--70. Amer. Math. Soc.,
  Providence, RI, 2009.

\bibitem[BH10]{BH10}
Cristiano Bocci and Brian Harbourne.
\newblock Comparing powers and symbolic powers of ideals.
\newblock {\em J. Algebraic Geom.}, 19(3):399--417, 2010.

\bibitem[BS93]{BS}
Mauro~C. Beltrametti and Andrew~J. Sommese.
\newblock On {$k$}-jet ampleness.
\newblock In {\em Complex analysis and geometry}, Univ. Ser. Math., pages
  355--376. Plenum, New York, 1993.

\bibitem[CEHH17]{CEHH}
Susan~M. Cooper, Robert J.~D. Embree, Huy~T\`ai H\`a, and Andrew~H. Hoefel.
\newblock Symbolic powers of monomial ideals.
\newblock {\em Proc. Edinb. Math. Soc. (2)}, 60(1):39--55, 2017.

\bibitem[CEL01]{CEL01}
Steven~Dale Cutkosky, Lawrence Ein, and Robert Lazarsfeld.
\newblock Positivity and complexity of ideal sheaves.
\newblock {\em Math. Ann.}, 321(2):213--234, 2001.

\bibitem[CHMR13]{CHMR13}
Ciro Ciliberto, Brian Harbourne, Rick Miranda, and Joaquim Ro\'{e}.
\newblock Variations of {N}agata's conjecture.
\newblock In {\em A celebration of algebraic geometry}, volume~18 of {\em Clay
  Math. Proc.}, pages 185--203. Amer. Math. Soc., Providence, RI, 2013.

\bibitem[DD21]{DD21}
Michael DiPasquale and Ben Drabkin.
\newblock On resurgence via asymptotic resurgence.
\newblock {\em J. Algebra}, 587:64--84, 2021.

\bibitem[DDSG{\etalchar{+}}18]{DDSGHN2018}
Hailong Dao, Alessandro De~Stefani, Elo\'{\i}sa Grifo, Craig Huneke, and Luis
  N\'{u}\~{n}ez Betancourt.
\newblock Symbolic powers of ideals.
\newblock In {\em Singularities and foliations. geometry, topology and
  applications}, volume 222 of {\em Springer Proc. Math. Stat.}, pages
  387--432. Springer, Cham, 2018.

\bibitem[Dem92]{Demailly}
Jean-Pierre Demailly.
\newblock Singular {H}ermitian metrics on positive line bundles.
\newblock In {\em Complex algebraic varieties ({B}ayreuth, 1990)}, volume 1507
  of {\em Lecture Notes in Math.}, pages 87--104. Springer, Berlin, 1992.

\bibitem[DFMS19]{DFMS}
Michael DiPasquale, Christopher~A. Francisco, Jeffrey Mermin, and Jay Schweig.
\newblock Asymptotic resurgence via integral closures.
\newblock {\em Trans. Amer. Math. Soc.}, 372(9):6655--6676, 2019.

\bibitem[DHN{\etalchar{+}}15]{DHNST15}
M.~Dumnicki, B.~Harbourne, U.~Nagel, A.~Seceleanu, T.~Szemberg, and
  H.~Tutaj-Gasi\'{n}ska.
\newblock Resurgences for ideals of special point configurations in
  {$\bold{P}^N$} coming from hyperplane arrangements.
\newblock {\em J. Algebra}, 443:383--394, 2015.

\bibitem[DHNT21]{DHNT21}
Le~Xuan Dung, Truong~Thi Hien, Hop~D. Nguyen, and Tran~Nam Trung.
\newblock Regularity and {K}oszul property of symbolic powers of monomial
  ideals.
\newblock {\em Math. Z.}, 298(3-4):1487--1522, 2021.

\bibitem[DHSTG14]{DHST14}
Marcin Dumnicki, Brian Harbourne, Tomasz Szemberg, and Halszka
  Tutaj-Gasi\'{n}ska.
\newblock Linear subspaces, symbolic powers and {N}agata type conjectures.
\newblock {\em Adv. Math.}, 252:471--491, 2014.

\bibitem[DV21]{DiPVill21}
Michael DiPasquale and Nelly Villamizar.
\newblock A lower bound for splines on tetrahedral vertex stars.
\newblock {\em SIAM J. Appl. Algebra Geom.}, 5(2):250--277, 2021.

\bibitem[EI95]{EI95}
J.~Emsalem and A.~Iarrobino.
\newblock Inverse system of a symbolic powers. {I}.
\newblock {\em J. Algebra}, 174(3):1080--1090, 1995.

\bibitem[ELS01]{ELS}
Lawrence Ein, Robert Lazarsfeld, and Karen~E. Smith.
\newblock {Uniform bounds and symbolic powers on smooth varieties}.
\newblock {\em Inventiones Math}, 144 (2):241--25, 2001.

\bibitem[Eva05]{Evain}
Laurent Evain.
\newblock On the postulation of {$s^d$} fat points in {$\Bbb P^d$}.
\newblock {\em J. Algebra}, 285(2):516--530, 2005.

\bibitem[Fek23]{Fekete}
M.~Fekete.
\newblock \"{U}ber die {V}erteilung der {W}urzeln bei gewissen algebraischen
  {G}leichungen mit ganzzahligen {K}oeffizienten.
\newblock {\em Math. Z.}, 17(1):228--249, 1923.

\bibitem[Ger96]{Ger95}
Anthony~V. Geramita.
\newblock Inverse systems of fat points: {W}aring's problem, secant varieties
  of {V}eronese varieties and parameter spaces for {G}orenstein ideals.
\newblock In {\em The {C}urves {S}eminar at {Q}ueen's, {V}ol. {X} ({K}ingston,
  {ON}, 1995)}, volume 102 of {\em Queen's Papers in Pure and Appl. Math.},
  pages 2--114. Queen's Univ., Kingston, ON, 1996.

\bibitem[GHMN17]{GHMN}
A.~V. Geramita, B.~Harbourne, J.~Migliore, and U.~Nagel.
\newblock Matroid configurations and symbolic powers of their ideals.
\newblock {\em Trans. Amer. Math. Soc.}, 369(10):7049--7066, 2017.

\bibitem[GHVT13]{GHV13}
Elena Guardo, Brian Harbourne, and Adam Van~Tuyl.
\newblock Asymptotic resurgences for ideals of positive dimensional subschemes
  of projective space.
\newblock {\em Adv. Math.}, 246:114--127, 2013.

\bibitem[GVV05]{GVV05}
Isidoro Gitler, Carlos Valencia, and Rafael~H. Villarreal.
\newblock A note on the {R}ees algebra of a bipartite graph.
\newblock {\em J. Pure Appl. Algebra}, 201(1-3):17--24, 2005.

\bibitem[Har77]{Hartshorne}
Robin Hartshorne.
\newblock {\em Algebraic geometry}.
\newblock Graduate Texts in Mathematics, No. 52. Springer-Verlag, New
  York-Heidelberg, 1977.

\bibitem[Har01]{Harbourne01}
Brian Harbourne.
\newblock On {N}agata's conjecture.
\newblock {\em J. Algebra}, 236(2):692--702, 2001.

\bibitem[HH02]{HoHu}
Melvin Hochster and Craig Huneke.
\newblock Comparison of symbolic and ordinary powers of ideals.
\newblock {\em Invent. Math.}, 147(2):349--369, 2002.

\bibitem[HR08]{HarbourneRoe08}
Brian Harbourne and Joaquim Ro\'{e}.
\newblock Discrete behavior of {S}eshadri constants on surfaces.
\newblock {\em J. Pure Appl. Algebra}, 212(3):616--627, 2008.

\bibitem[HS06]{HS06}
Craig Huneke and Irena Swanson.
\newblock {\em Integral closure of ideals, rings, and modules}, volume 336 of
  {\em London Mathematical Society Lecture Note Series}.
\newblock Cambridge University Press, Cambridge, 2006.

\bibitem[HT21]{Ha-Nguyen-Okounkov}
Huy~Tai {Ha} and Thai {Thanh Nguyen}.
\newblock {Newton-Okounkov body, Rees algebra, and analytic spread of graded
  families of monomial ideals}.
\newblock {\em arXiv e-prints}, page arXiv:2111.00681, October 2021.

\bibitem[HTW20]{Hernandez-Teixera-Witt-2020}
Daniel~J. Hern\'{a}ndez, Pedro Teixeira, and Emily~E. Witt.
\newblock Frobenius powers.
\newblock {\em Math. Z.}, 296(1-2):541--572, 2020.

\bibitem[Hun06]{HunekeAIM}
Craig Huneke.
\newblock Open problems on powers of ideals.
\newblock
  \href{https://www.aimath.org/WWN/integralclosure/Huneke.pdf}{https://www.aimath.org/WWN/integralclosure/Huneke.pdf},
  2006.

\bibitem[Iar97]{iarrobino97}
A.~Iarrobino.
\newblock Inverse system of a symbolic power. {III}. {T}hin algebras and fat
  points.
\newblock {\em Compositio Math.}, 108(3):319--356, 1997.

\bibitem[IK99]{Iarrobino-Kanev-1999}
Anthony Iarrobino and Vassil Kanev.
\newblock {\em Power sums, {G}orenstein algebras, and determinantal loci},
  volume 1721 of {\em Lecture Notes in Mathematics}.
\newblock Springer-Verlag, Berlin, 1999.
\newblock Appendix C by Iarrobino and Steven L. Kleiman.

\bibitem[Joh11]{Johnson11}
Keith Johnson.
\newblock Super-additive sequences and algebras of polynomials.
\newblock {\em Proc. Amer. Math. Soc.}, 139(10):3431--3443, 2011.

\bibitem[Laz04]{Laz04}
Robert Lazarsfeld.
\newblock {\em Positivity in algebraic geometry. {II}}, volume~49 of {\em
  Ergebnisse der Mathematik und ihrer Grenzgebiete. 3. Folge. A Series of
  Modern Surveys in Mathematics [Results in Mathematics and Related Areas. 3rd
  Series. A Series of Modern Surveys in Mathematics]}.
\newblock Springer-Verlag, Berlin, 2004.
\newblock Positivity for vector bundles, and multiplier ideals.

\bibitem[Mac94]{Macaulay}
F.~S. Macaulay.
\newblock {\em The algebraic theory of modular systems}.
\newblock Cambridge Mathematical Library. Cambridge University Press,
  Cambridge, 1994.
\newblock Revised reprint of the 1916 original, With an introduction by Paul
  Roberts.

\bibitem[Man20]{Mantero-2020}
Paolo Mantero.
\newblock The structure and free resolutions of the symbolic powers of star
  configurations of hypersurfaces.
\newblock {\em Trans. Amer. Math. Soc.}, 373(12):8785--8835, 2020.

\bibitem[MS18]{MaSchwede}
Linquan Ma and Karl Schwede.
\newblock Perfectoid multiplier/test ideals in regular rings and bounds on
  symbolic powers.
\newblock {\em Invent. Math.}, 214(2):913--955, 2018.

\bibitem[Nag59]{Nagata59}
Masayoshi Nagata.
\newblock On the {$14$}-th problem of {H}ilbert.
\newblock {\em Amer. J. Math.}, 81:766--772, 1959.

\bibitem[NT19]{NT19}
Uwe Nagel and Bill Trok.
\newblock Interpolation and the weak {L}efschetz property.
\newblock {\em Trans. Amer. Math. Soc.}, 372(12):8849--8870, 2019.

\bibitem[{\O}st06]{Osterdal06}
Lars~Peter {\O}sterdal.
\newblock Subadditive functions and their (pseudo-)inverses.
\newblock {\em J. Math. Anal. Appl.}, 317(2):724--731, 2006.

\bibitem[Swa97]{Swanson}
Irena Swanson.
\newblock Powers of ideals. {P}rimary decompositions, {A}rtin-{R}ees lemma and
  regularity.
\newblock {\em Math. Ann.}, 307(2):299--313, 1997.

\bibitem[TG03]{Tutaj}
Halszka Tutaj-Gasi\'{n}ska.
\newblock A bound for {S}eshadri constants on {${\Bbb P}^2$}.
\newblock {\em Math. Nachr.}, 257:108--116, 2003.

\bibitem[Wal21]{walker}
Robert~M. Walker.
\newblock The symbolic generic initial system of an a.s.p. ideal.
\newblock {\em preprint}, 2021.

\bibitem[Xu94]{Xu}
Geng Xu.
\newblock Curves in {${\bf P}^2$} and symplectic packings.
\newblock {\em Math. Ann.}, 299(4):609--613, 1994.

\bibitem[Zar49]{Zariski-1949}
Oscar Zariski.
\newblock A fundamental lemma from the theory of holomorphic functions on an
  algebraic variety.
\newblock {\em Ann. Mat. Pura Appl. (4)}, 29:187--198, 1949.

\end{thebibliography}

\appendix

\section{Formulas involving differentiation and contraction}\label{app:formulas}

In this appendix we collect, for the convenience of the reader, proofs of some of the formulas that we use in~\Cref{sec:InverseSystems}.  Let $R=\kk[x_0,\ldots,x_N]$ be a polynomial ring, $D_R$ the ring of $\kk$-linear differential operators on $R$, $\mathcal{D}$ the divided power algebra on the divided power monomials $Y^{[\mathbf{a}]}$, with $\ba\in\Z^{N+1}_{\ge 0}$, and $S$ the polynomial ring $\kk[y_0,\ldots,y_N]$.  We have the action of $R$ on $\mathcal{D}$ by contraction, written $\contract$, and $R$ on $S$ by partial differentiation, written $\circ$.  First we prove the higher order product rule~\eqref{eq:higherproductrule}.

\begin{lem}\label{lem:HigherProductRule}
Let $f,g\in R$, $i$ be an integer between $0$ and $N$, and $k\ge 1$ an integer.  In characteristic $0$ we have
\[
\frac{\partial^k (fg)}{\partial x_i^k}=\sum_{j=0}^k \binom{k}{j} \frac{\partial^j f}{\partial x_i^j}\frac{\partial^{k-j} g}{\partial x_i^{k-j}}
\]
In arbitrary characteristic we have
\[
D_{k\be_i}(fg)=\sum_{j=0}^k D_{j\be_i}(f)D_{(k-j)\be_i}(g).
\]
\end{lem}
\begin{proof}
The first formula follows from induction and the ordinary product rule.  It also follows from the second via the identification $D_{\ba}=\frac{1}{\ba!}\frac{\partial^{\ba}}{\partial x^\ba}$, so we prove the second.  We start by proving the formula for $f=x_i^m$ and $g=x_i^n$ where $m,n$ are integers.  Then $D_{k\be_i}(x^{m+n})=\binom{m+n}{k}x_i^{m+n-k}$ and 
\begin{align*}
\sum_{j=0}^k D_{j\be_i}(x_i^m)D_{(k-j)\be_i}(x_i^n) & =\sum_{j=0}^k \binom{m}{j}x_i^{m-j}\binom{n}{k-j}x_i^{n-k+j}\\
&=\left(\sum_{j=0}^k \binom{m}{j}\binom{n}{k-j}\right)x_i^{m+n-k}\\
&=\binom{m+n}{k}x_i^{m+n-k},
\end{align*}
where in the identity above, if either $j>m$ or $k-j>n$, we interpret $x_i^{m-j}=0$ or $x_i^{n-k+j}=0$, respectively.  The binomial coefficients are also interpreted in this way.

Now suppose that $f=x^\ba=x_i^{m}x^{\ba'}$ and $g=x^\bb=x_i^nx^{\bb'}$ for $\ba,\bb\in\Z^{N+1}_{\ge 0}$, where $x^{\ba'}$ and $x^{\bb'}$ are not divisible by $x_i$.  Then $D_{k\be_i}(fg)=x^{\ba'+\bb'}D_{k\be_i}(x_i^{m+n})$ and likewise $D_{j\be_i}(x^\ba)D_{(k-j)\be_i}(x^\bb)=x^{\ba'+\bb'}D_{j\be_i}(x^m_i)D_{(k-j)\be_i}(x^n_i)$ for $j=0,\ldots,k$.  Since the same factor of $x^{\ba'+\bb'}$  pulls out of both sides of the identity, it reduces to what we have already shown.  To get the result where $f$ is an arbitrary polynomial and $g$ is a monomial, we use linearity of the differential operators in $f$.  Finally, to get the full result we use linearity in $g$.
\end{proof}

\begin{lem}\label{lem:prodrulechar0}
Suppose $g\in S$ is a homogeneous polynomial.  Let $F\in R$ be homogeneous of degree $d\geq 1$.  In characteristic $0$, we have
\[
F\circ (y_jg) = \frac{\partial F}{\partial x_j}\circ g + y_j\left(F\circ g\right)
\]
for every $j=0,\ldots,N$.
\end{lem}
\begin{proof}
Suppose $F$ is a monomial.  We induct on the exponent of $x_j$ in $F$.  First suppose that the exponent of $x_j$ in $F$ is $0$.  In this case, $y_j$ acts as a constant as far as differentiation by $F$ is concerned and thus
$
F\circ(y_jg)=y_j(F\circ g).
$
Since we also have $(\partial F)/(\partial x_j)=0$, this proves the lemma when the exponent of $x_j$ in $F$ is $0$.  Now suppose that the exponent on $x_j$ is positive.  Then we can write $F=x_jF_0$ for some monomial $F_0$.  We have
\begin{equation}\label{eq:first_reduction}
F\circ (y_jg)=(F_0x_j)\circ(y_jg)=F_0\circ (x_j\circ (y_jg))=F_0\circ g+F_0\circ (y_j(x_j\circ g)),
\end{equation}
where the last equality follows from the product rule.  Since the exponent of $x_j$ in $F_0$ is one less than the exponent of $x_j$ in $F$, our induction hypothesis yields
\[
F_0\circ(y_j(x_j\circ g))=\frac{\partial F_0}{\partial x_j}\circ (x_j\circ g)+y_j(F_0\circ (x_j\circ g))=\frac{\partial F_0}{\partial x_j}\circ (x_j\circ g)+y_j(F\circ g).
\]
Substituting this in to the last equality in~\eqref{eq:first_reduction} yields
\begin{align*}
F\circ (y_jg) & =F_0\circ g+F_0\circ (y_j(x_j\circ g))\\
&=F_0\circ g+\frac{\partial F_0}{\partial x_j}\circ (x_j\circ g)+y_j(F\circ g)\\
& =\left(F_0+x_j\frac{\partial F_0}{\partial x_j}\right)\circ g+y_j(F\circ g)\\
& = \frac{\partial F}{\partial x_j}\circ g+y_j(F\circ g).
\end{align*}
This proves the lemma when $F$ is a monomial.  The general result follows from linearity of the derivative.
\end{proof}

\begin{lem}\label{lem:prodruledivided}
Suppose $g\in \mathcal{D}$ is a divided power homogeneous polynomial (that is, all of it divided power monomials have the same degree).  Let $F\in R$ be homogeneous.  In arbitrary characteristic, we have
\[
F\contract (Y_jg) = D_{\be_j}(F)\contract g + Y_j\left(F\contract g\right)
\]
for every $j=0,\ldots,N$.
\end{lem}
\begin{proof}
First, we show that the formula holds when $F=x_j^{m}$ and $g=Y_j^{[n]}$. In fact, both sides are $0$ if $n\le m-2$, and if $n=m-1$, $F\contract (Y_jg) =  x_j^{m} \contract ((n+1)Y_j^{[n+1]}) = m=m+0= mx_j^{m-1} \contract (Y_j^{[m-1]})+0 = D_{\be_j}(F)\contract g+Y_j(F\contract g)$. Otherwise, 
\[
F\contract (Y_jg)=(n+1)Y_j^{[n+1-m]}=mx_j^{m-1} \contract (Y_j^{[n]}) + Y_jY_j^{[n-m]}=D_{\be_j}(F)\contract g+Y_j(F\contract g).
\]
Now suppose $F=x_j^m$ and $g=Y^{[\bb]}=Y^{[\bb']}Y_j^{[n]}$, where $Y^{[\bb']}$ is not divisible by $Y_j$.  Then
\begin{multline*}
F\contract (Y_jg)=Y^{[\bb']}(F\contract (Y_jY_j^{[n]}))
\\=Y^{[\bb']}(D_{\be_j}(F)\contract Y_j^{[n]}) + Y^{[\bb']}Y_j(F\contract Y_j^{[n]})=D_{\be_j}(F)\contract Y^{[\bb]}+Y_j(F\contract Y^{[\bb]}),
\end{multline*}
since we have proved the identity for $g=Y_j^{[n]}$ and we can pull $Y^{[\bb']}$ in and out of the contraction with $F$ because $Y^{[\bb']}$ acts like a constant under contraction with $F$.  Now suppose $F=x^{\ba}=x^{\ba'}x_j^{m}$, where $x^{\ba'}$ is not divisible by $x_j$, and $g=Y^{[\bb]}$.  Then
\begin{multline*}
	F\contract (Y_jg)=x^{[\bb']}\contract(x_j^{m}\contract (Y_jg))
	\\=x^{[\bb']}\contract (D_{\be_j}(x_j^{m})\contract g) + x^{\bb'}\contract (Y_j(x_j^{m}\contract g)=D_{\be_j}(x^{\bb})\contract g+Y_j(x^{\bb}\contract g),
\end{multline*}
since we have proved the identity for any divided power monomial $g=Y^{[\bb]}$, contraction is linear, differentiation with respect to $x_j$ commutes with $x^{\bb'}$, and contraction by $x^{\bb'}$ commutes with $Y_j$ because $x^{\bb'}$ is not divisible by $x_j$.  Thus the desired equality holds if $F$ is a monomial (if $m=0$ we interpret $x_j^{m-1}$ as $0$, not $x_j^{-1}$) and $g$ is a monomial. The full result follows from linearity of the contraction.
\end{proof}

\begin{lem}\label{lem:prodrulekp}
Suppose $F\in R$ and $g\in S$ are both homogeneous.  In characteristic $0$,
\[
F\circ (y_j^k g)=\displaystyle \sum_{i=0}^k \binom{k}{i} y_j^{k-i}\left( \frac{\partial^i F}{\partial x_j^i}\circ g\right)
\]
for every $j=0,\ldots,N$ and every $k \in \N$.  If $g\in\mathcal{D}$ is homogeneous then we have, in arbitrary characteristic,
\[
F\contract (Y_j^{[k]}g) = \displaystyle \sum_{i=0}^k Y_j^{[k-i]}(D_{i\mathbf{e_j}}(F)\contract g)
\]
for every $j=0,\ldots,N$ and every $k \in \N$.
\end{lem}
\begin{proof}
In characteristic $0$, both statements can be proven by induction on $k$, where the base case is \Cref{lem:prodrulechar0} for $S$ and \Cref{lem:prodruledivided} for $\mathcal{D}$.  We leave the details to the interested reader.  The strongest statement is the second in arbitrary characteristic, and we prove this one.  (Note that the first statement also follows from the second in characteristic $0$ by the identification $D_{i\mathbf{e_j}}(F)=\frac{1}{i!}\frac{\partial^i F}{\partial x_j^i}$ and the $R$-module isomorphism between $S$ and $\mathcal{D}$.)

We start by proving the second statement when $F=x_j^{m}$ and $g=Y_j^{[n]}$ for non-negative integers $m$ and $n$. On the one hand, we have 
\begin{equation}\label{eq:MonLHS}
F\contract (Y_j^{[k]}g) = {k+n\choose k} x_j^{m} \contract Y_j^{[k+n]} = {k+n\choose k} Y_j^{[k+n-m]}.
\end{equation}
On the other hand, 
\begin{align}\label{eq:MonRHS}
\sum_{i=0}^k Y^{[k-i]}(D_{i\mathbf{e_j}}(F)\contract g)
&= \sum_{i=0}^k Y^{[k-i]} \left( {m \choose i} x_j^{m-i} \circ Y_j^{[n]} \right) \nonumber \\
&= \sum_{i=0}^k {m \choose i} Y^{[k-i]} Y_j^{[n-m+i]} \\
&= \sum_{i=0}^k {m \choose i}{k+n-m \choose k-i} Y_j^{[k+n-m]}. \nonumber
\end{align}
In the above sum, the terms when $i > m$ (corresponding to $D_{i\mathbf{e_j}}(F)$ = 0) or when $k-i>k+n-m$ (corresponding to $n<m-i$, hence $x_j^{m-i} \contract y_j^{[n]}=0$) are $0$. The lemma holds from the combinatorial identity
\[
\sum_{i=0}^k {m \choose i}{k+n-m \choose k-i} = {k+n \choose k}.
\]
Suppose $F=x_j^m$ and $g=Y^{[\bb]}=Y^{[\bb']}Y_j^{[n]}$.  Then the factor $Y^{[\bb']}$ will pull out of~\eqref{eq:MonLHS} and of every summand in~\eqref{eq:MonRHS}.  Thus the result follows from what has been shown.  Now suppose $F=x^{\ba}=x^{\ba'}x_j^{m}$, where $x^{\ba'}$ is not divisible by $x_j$, and $g=Y^{[\bb]}$.  Then $F\contract (Y_j^{[k]})=x^{\ba'}\contract (x_j^{m}\contract (Y_j^{[k]}g))=x^{\ba'}\contract\left(\sum_{i=0}^k Y_j^{[k-i]}(D_{i\mathbf{e_j}}(x_j^m)\contract g)\right)$.  Now
\begin{align*}
x^{\ba'}\contract\left(\sum_{i=0}^k Y_j^{[k-i]}(D_{i\mathbf{e_j}}(x_j^m)\contract g)\right) & =\sum_{i=0}^k x^{\ba'}\contract \left(Y_j^{[k-i]}(D_{i\mathbf{e_j}}(x_j^m)\contract g)\right)\\
&=\sum_{i=0}^k Y_j^{[k-i]}\left(x^{\ba'}\contract(D_{i\mathbf{e_j}}(x_j^m)\contract g)\right)\\
& =\sum_{i=0}^k Y_j^{[k-i]}\left((x^{\ba'}D_{i\mathbf{e_j}}(x_j^m))\contract g\right)\\
& =\sum_{i=0}^k Y_j^{[k-i]}\left(D_{i\mathbf{e_j}}(x^{\ba'}x_j^m)\contract g\right)\\
& =\sum_{i=0}^k Y_j^{[k-i]}\left(D_{i\mathbf{e_j}}(x^{\ba})\contract g\right),
\end{align*}
where the first equality follows by linearity of contraction, the second because $x^{\ba'}$ does not involve the variable $x_j$, the third by the definition of contraction, and the fourth also because $x^{\ba'}$ does not involve the variable $x_j$.  This proves the identity when $F$ and $g$ are monomials.  The identity now follows when $F$ and $g$ are polynomials by linearity.
\end{proof}

\section{The inverse system of powers of the ideal of a point}
Emsalem and Iarrobino show in~\cite{EI95} that
the fundamental computation when finding the inverse system of the symbolic powers of a variety is finding the inverse system of the symbolic powers of the ideal of a single point.  We revisit this computation using~\Cref{lem:prodrulekp}.  Let $p=[b_0:b_1:\ldots:b_N]\in\mathbb{P}^N$ and
\[
\fm_p=\langle b_1x_0-b_0x_1,\ldots,b_Nx_0-b_0x_N\rangle\subset R=\kk[x_0,\ldots,x_N].
\]
be the ideal of homogeneous polynomials vanising on $p$.  We write $L_p=b_0y_0+\ldots+b_Ny_N\in S$ for the dual linear form.  An important observation in~\cite{EI95} is that, if $F\in R$ is homogeneous of degree $d\le k$, then
\begin{equation}\label{eq:char0power}
F\circ L_p^k=\frac{k!}{(k-d)!}L_p^{k-d}F(p),
\end{equation}
where $F(p)$ is the evaluation of $F$ at $p$.

In arbitrary characteristic, we also let $L_p$ denote the dual linear form $b_0Y_0+\cdots+b_NY_N\in\mathcal{D}$, relying on context to differentiate between $L_p\in S$ and $L_p\in\mathcal{D}$.  In $\mathcal{D}$, we define the \textit{divided power} of $L_p$ by
$L_p^{[k]}=\sum_{|\ba|=k} b_0^{a_0}\cdots b_N^{a_N}Y^{[\ba]}.$

The definition of $L_p^{[k]}$ is made precisely so that the analog of~\eqref{eq:char0power} holds.  Namely, if $F\in R$ is homogeneous of degree $d\le k$, then
\begin{equation}
\label{eq:charpPower}
F\contract L_p^{[k]}=L_p^{[k-d]}F(p),
\end{equation}
where again $F(p)$ is the evaluation of $F$ at $p$.  Both~\eqref{eq:char0power} and~\eqref{eq:charpPower} follow from a direct computation.  The following result is shown in~\cite{EI95} (see also~\cite{Ger95}).

\begin{lem}
\label{lem:OnePointInverseSystem}
In characteristic $0$,
\[
(\fm_p^n)^\perp_d=
\left\lbrace
\begin{array}{ll}
S_d & \mbox{if } d<n\\
\langle L_p^{d-n+1} \rangle_{d} & \mbox{if } d\ge n
\end{array}.
\right.
\]
In arbitrary characteristic,
\[
(\fm_p^n)^\perp_d=
\left\lbrace
\begin{array}{ll}
\mathcal{D}_d & \mbox{if } d<n\\
{\rm span}\{ Y^{[\ba]}L_p^{[c]}: d-n+1\le c \le d, |\ba|=d-c \} & \mbox{if } d\ge n
\end{array}.
\right.
\]
\end{lem}
\begin{proof}
We prove the formula for the action of $R$ on $\mathcal{D}$.  The case when $d<n$ is clear, so we assume that $d\ge n$.  It is straightforward to show that, when $d\ge n$,
\[
\dim (\fm_p^n)_d=\binom{d+N+1}{N+1}-\binom{n+N}{N+1} \quad \mbox{and hence} \quad
\dim (\fm_p^n)^\perp_d=\binom{n+N}{N+1}.
\]
Examining the terms of $Y^{[\ba]}L_p^{[c]}$, we see that, for some $0\le i\le N$, the divided monomials of the form
\[
\{Y_i^{[c]}Y^{[\ba']}: Y_i \mbox{ does not appear in } Y^{[\ba']}, d-n+1\le c\le d, \mbox{ and } c+|\ba'|=d\}
\]
all appear as a term in some $Y^{[\ba]}L^{[c]}$ on the right hand side.  There are $\binom{n+N}{N+1}$ of these monomials, thus the dimension of the right hand side is at least the dimension of $\dim (\fm_p^n)^\perp_d$.  Thus it suffices to show that $Y^{[\ba]}L_p^{[c]}\in (\fm_p^n)_d^\perp$ for $d-n+1\le c\le d$ and $|\ba|=d-c$.  For this we take a form $F\in (\fm_p^n)_d$ and show that $F\contract (Y^{[\ba]}L_p^{[c]})=0$.

We induct on $n$ and $|\ba|$.  If $n=1$ or $|\ba|=0$ then $c=d$ and $F\contract L_p^{[c]}=F\contract L_p^{[d]}=F(p)$ by~\eqref{eq:charpPower}.  Since $F\in \fm_p, F(p)=0$ and we are done.  Now suppose $n>1$ and $|\ba|>0$.  Then, for some $0\le i\le N+1$, we can write $Y^{[\ba]}=Y_i^{[k]}Y^{[\ba']}$ where $0<k\le d-c$ and $Y_i$ does not appear in $Y^{[\ba']}$.  By~\Cref{lem:prodrulekp},
\begin{equation}
\label{eq:PowerPoint}
F\contract(Y^{[\ba]}L_p^{[c]})=\sum_{j=0}^k Y_i^{[k-j]}(D_{j\be_i}(F)\contract Y^{[\ba']}L_p^{[c]} ).
\end{equation}
Note that if $j=0$ then $D_{0\be_i}(F)=F$ and $F\contract Y^{[\ba']}L_p^{[c]}=0$ by induction on $|\ba|$.  If $1<j\le k$ then $D_{j\be_i}(F)\in \fm_p^{n-j}$ by~\Cref{ex:PowersDiffClosed} and thus $D_{j\be_i}(F)\contract Y^{[\ba']}L_p^{[c]}=0$ by induction on $n$.  So all terms in~\eqref{eq:PowerPoint} vanish and we are done.

An identical strategy can be used to show the formula for $(\fm_p^n)^\perp_d$ for the action of $R$ on $S$; the proof can be simplified a little using~\Cref{lem:prodrulechar0} instead of~\Cref{lem:prodrulekp}.
\end{proof}

\begin{rem}
A different proof of \Cref{lem:OnePointInverseSystem} relies on the $GL_{N+1}$-equivariance of the differentiation and contraction actions (see~\cite[Proposition~A.3]{Iarrobino-Kanev-1999}), under which we may assume that $p=[1:0:\cdots:0]$.
\end{rem}

\end{document}